\documentclass[12pt]{article}

\usepackage{enumerate}
\usepackage{amsthm}
\usepackage{amssymb}
\usepackage{amsmath}
\usepackage{amsfonts}
\usepackage{cancel}
\usepackage{times}
\usepackage{bbm}
\usepackage{cite}
\usepackage{color, soul}
\usepackage{hyperref}
\hypersetup{
	colorlinks   = true,
	citecolor    = black,
	linkcolor    = blue
}

\usepackage[paper=a4paper, top=2.5cm, bottom=2.5cm, left=2.5cm, right=2.5cm]{geometry}

\def\qed{\hfill $\vcenter{\hrule height .3mm
		\hbox {\vrule width .3mm height 2.1mm \kern 2mm \vrule width .3mm
			height 2.1mm} \hrule height .3mm}$ \bigskip}

\def \RR {\mathbb R}
\def \NN {\mathbb N}
\def \EE {\mathbb E}
\def \ZZ {\mathbb Z}

\def \PP {\mathbb P}

\def \vphi {\varphi}

\def \FF {\mathcal{F}}

\newcommand\norm[1]{\left\lVert#1\right\rVert}
\newcommand\inner[2]{\langle#1,#2\rangle}
\newtheorem{theorem}{Theorem}
\newtheorem{lemma}{Lemma}

\newtheorem{proposition}{Proposition}

\theoremstyle{definition}
\newtheorem{definition}[theorem]{Definition}
\theoremstyle{remark}
\newtheorem{remark}[theorem]{Remark}
\newtheorem*{remark*}{Remark}

\long\def\symbolfootnotetext[#1]#2{\begingroup
	\def\thefootnote{\fnsymbol{footnote}}\footnotetext[#1]{#2}\endgroup}

\newcommand{\Ent}{\mathrm{Ent}}
\newcommand{\Cov}{\mathrm{Cov}}
\newcommand{\Id}{\mathrm{I}_d}

\begin{document}
\title{The CLT in high dimensions: quantitative bounds via martingale embedding}
\author{Ronen Eldan\thanks{Weizmann Institute of Science. Incumbent of the Elaine Blond career development chair. Partially supported by Israel Science Foundation grant 715/16.},~ Dan Mikulincer\thanks{Weizmann Institute of Science. Supported by an Azrieli Fellowship award from the Azrieli Foundation.}~~and Alex Zhai\thanks{Stanford University. Partially supported by a Stanford Graduate Fellowship.}}
\maketitle

\begin{abstract}
We introduce a new method for obtaining quantitative convergence rates for the central limit theorem (CLT) in a high dimensional setting. Using our method, we obtain several new bounds for convergence in transportation distance and entropy, and in particular: (a) We improve the best known bound, obtained by the third named author \cite{zhai2016multivariate}, for convergence in quadratic Wasserstein transportation distance for bounded random vectors; (b) We derive the first non-asymptotic convergence rate for the entropic CLT in arbitrary dimension, for general log-concave random vectors (this adds to \cite{courtade2017existence}, where a finite Fisher information is assumed); (c) We give an improved bound for convergence in transportation distance under a log-concavity assumption and improvements for both metrics under the assumption of strong log-concavity. Our method is based on martingale embeddings and specifically on the Skorokhod embedding constructed in \cite{eldan2016skorokhod}.
\end{abstract}

\section{Introduction}

Let $X^{(1)}, \ldots , X^{(n)}$ be i.i.d. random vectors in $\RR^d$. By
the central limit theorem, it is well-known that, under mild
conditions, the sum $\frac{1}{\sqrt{n}} \sum_{i = 1}^n X^{(i)}$
converges to a Gaussian. With $d$ fixed, there is an extensive
literature showing that the distance from Gaussian under various
metrics decays as $\frac{1}{\sqrt{n}}$ as $n \rightarrow \infty$, and
this is optimal.

However, in high-dimensional settings, it is often the case that the
dimension $d$ is not fixed but rather grows with $n$. It then becomes
necessary to understand how the convergence rate depends on dimension,
and the optimal dependence here is not well understood. We present a
new technique for proving central limit theorems in $\RR^d$ that is
suitable for establishing quantitative estimates for the convergence
rate in the high-dimensional setting. The technique, which is
described in more detail in Section \ref{sec: method description} below, is based on pathwise analysis: we first couple the random vector with a Brownian motion via a martingale embedding. This gives rise to a coupling between the sum and a Brownian motion for which we can establish bounds on the concentration of the quadratic variation. We use a multidimensional version of a Skorokhod embedding, inspired by a construction of the first named author from
\cite{eldan2016skorokhod}, as a manifestation of the martingale embedding.

Using our method, we prove new bounds on quadratic \emph{transportation}
(also known as ``Kantorovich'' or ``Wasserstein'') distance in the
CLT, and in the case of log-concave distributions, we also give bounds
for \emph{entropy} distance. Let $\mathcal{W}_2(A, B)$ denote the quadratic
transportation distance between two $d$-dimensional random vectors $A$
and $B$. That is,
$$\mathcal{W}_2\left(A,B\right) = \sqrt{\inf\limits_{\stackrel{(X,Y)\ s.t.}{
X \sim A,\ Y \sim B}}\EE\left[\norm{X-Y}_2^2\right]},$$
where the infimum is taken over all couplings of the vectors $A$ and $B$.
As a first demonstration of our method, we begin with an improvement to the best known convergence rate in the case of bounded random vectors.
\begin{theorem} \label{thm:w2-bounded}
  Let $X$ be a random $d$-dimensional vector. Suppose that $\EE[X] =
  0$ and $\norm{X} \leq \beta$ almost surely for some $\beta > 0$. Let
  $\Sigma = \Cov(X)$, and let $G \sim \mathcal{N}(0, \Sigma)$ be a
  Gaussian with covariance $\Sigma$. If $\{X^{(i)}\}_{i=1}^n$ are
  i.i.d copies of $X$ and $S_n = \frac{1}{\sqrt{n}}\sum_{i=1}^n
  X^{(i)}$, then
  \[ \mathcal{W}_2(S_n, G) \leq \frac{\beta\sqrt{d}\sqrt{32+2\log_2(n)}}{\sqrt{n}}. \]
\end{theorem}

Theorem \ref{thm:w2-bounded} improves a result of the third named
author \cite{zhai2016multivariate} that gives a bound of order
$\frac{\beta \sqrt{d} \log n}{\sqrt{n}}$ under the same conditions. It
was noted in \cite{zhai2016multivariate} that when $X$ is supported on
a lattice $\beta \ZZ^d$, then the quantity $\mathcal{W}_2(S_n, G)$ is
of order $\frac{\beta \sqrt{d}}{\sqrt{n}}$. Thus, Theorem
\ref{thm:w2-bounded} is within a $\sqrt{\log n}$ factor of optimal.

When the distribution of $X$ is isotropic and log-concave, we can
improve the bounds guaranteed by Theorem
\ref{thm:w2-bounded}. In this case, however, a more general bound has already been established in \cite{courtade2017existence}, see discussion below.

\begin{theorem} \label{thm:w2-log-concave}
  Let $X$ be a random $d$-dimensional vector. Suppose that the
  distribution of $X$ is log-concave and isotropic. Let $G \sim \mathcal{N}(0, \Id)$ be a standard Gaussian. If
  $\{X^{(i)}\}_{i=1}^n$ are i.i.d copies of $X$ and $S_n =
  \frac{1}{\sqrt{n}}\sum\limits_{i=1}^n X^{(i)}$, then there exists a
  universal constant $C>0$ such that, if $d \geq 8$,
  \[ \mathcal{W}_2(S_n, G) \leq \frac{Cd^{3/4}\ln(d)\sqrt{\ln(n)}}{\sqrt{n}}. \]
\end{theorem}
\begin{remark}
  We actually prove the slightly stronger bound
  \[ \mathcal{W}_2(S_n, G) \leq \frac{C\kappa_d\ln(d)\sqrt{d\ln(n)}}{\sqrt{n}}, \]
  where
  \begin{equation} \label{kappa_d}
  \kappa_d := \sup_{\substack{\text{$\mu$ isotropic,} \\ \text{log-concave}}}\Big\lVert\int\limits_{\RR^d}x_1x\otimes x\mu(dx)\Big\rVert_{HS},
  \end{equation}
  as defined in \cite{eldan2013thin}. Results in \cite{eldan2013thin}
  and \cite{lee2016eldan} imply that $\kappa_d = O(d^{1/4})$, leading
  to the bound in Theorem \ref{thm:w2-log-concave}. If the
  \emph{thin-shell conjecture} (see \cite{anttila2003central}, as well \cite{bobkov2003central}) is true, then the bound is improved to
  $\kappa_d = O(\sqrt{\ln(d)})$, which yields
  \[ \mathcal{W}_2(S_n, G) \leq \frac{C\sqrt{d\ln(d)^3\ln(n)}}{\sqrt{n}}. \]
  By considering, for example, a random vector uniformly distributed on the unit cube, one can see that the above bound is sharp up to the logarithmic factors.
\end{remark}
\begin{remark}
  To compare with the previous theorem, note that if $\Cov(X) = \Id$,
  then $\EE \norm{X}^2 = d$. Thus, in applying Theorem
  \ref{thm:w2-bounded} we must take $\beta \ge \sqrt{d}$, and the
  resulting bound is then of order at least $\frac{d \sqrt{\log n}}{\sqrt{n}}$.
\end{remark}

Next, we describe our results regarding convergence rate in entropy. If $A$ and $B$ are random vectors such that $A$ has density $f$ with respect to the law of $B$, then relative entropy of $A$ with respect to $B$ is given  by
$$\mathrm{Ent}\left(A||B\right)  = \EE\left[\ln\left(f(A)\right)\right].$$
As a warm-up, we first use our method to recover the entropic CLT in any fixed dimension. In dimension one this was first established by Barron, \cite{barron1986entropy}. The same methods may also be applied to prove a multidimensional analogue. See \cite{bobkov2013rate} for  a more quantitative version of the theorem.
\begin{theorem} \label{thm:entropic clt}
	Suppose that $\mathrm{Ent}\left(X||G\right) < \infty$. Then one has
	$$\lim\limits_{n \to \infty}\mathrm{Ent}(S_n||G) = 0.$$
\end{theorem}
The next result gives the first non-asymptotic convergence rate for the entropic CLT, again under the log-concavity assumption (other non-asymptotic results appear in previous works, notably \cite{courtade2017existence}, but require additional assumptions; see below).
\begin{theorem} \label{thm:ent-log-concave}
   Let $X$ be a random $d$-dimensional vector. Suppose that the
  distribution of $X$ is log-concave and isotropic. Let $G \sim \mathcal{N}(0,
  \mathrm{I}_d)$ be a standard Gaussian. If $\{X^{(i)}\}_{i=1}^n$ are i.i.d copies of $X$ and $S_n
  = \frac{1}{\sqrt{n}}\sum\limits_{i=1}^n X^{(i)}$ then
  \[ \Ent(S_n||G) \leq \frac{C d^{10}(1 +  \Ent(X||G))}{n}, \]
  for a universal constant $C > 0$.
\end{theorem}



Our method also yields a different (and typically stronger) bound if
the distribution is strongly log-concave.

\begin{theorem} \label{thm:ent-strong-log-concave}
  Let $X$ be a $d$-dimensional random vector with $\EE[X] = 0$ and
  $\Cov(X) = \Sigma$. Suppose further that $X$ is $1$-uniformly log concave (i.e. it has a probability density
  $e^{-\varphi(x)}$ satisfying $\nabla^2 \varphi \succeq \Id$) and that
  $\Sigma \succeq \sigma\mathrm{I}_d$ for some $\sigma > 0$.

  Let $G \sim \mathcal{N}(0, \Sigma)$ be a Gaussian with the same covariance as $X$ and let $\gamma \sim \mathcal{N}(0, \mathrm{I}_d)$ be a standard Gaussian. If
  $\{X^{(i)}\}_{i=1}^n$ are i.i.d copies of $X$ and $S_n =
  \frac{1}{\sqrt{n}}\sum\limits_{i=1}^n X^{(i)}$, then
  \[ \Ent(S_n || G) \leq \frac{2\left(d + 2\mathrm{Ent}\left(X||\gamma\right)\right)}{\sigma^4n}. \]
\end{theorem}
\begin{remark}
The theorem can be applied when $X$ is isotropic and  $\sigma$-uniformly log concave for some $\sigma > 0$. In this case, a change of variables shows that $\sqrt{\sigma}X$ is $1$-uniformly log concave and has $\sigma\mathrm{I}_d$ as a covariance matrix. Since relative entropy to a Gaussian is invariant under affine transformations, if $G \sim \mathcal{N}(0,\mathrm{I}_d)$ is a standard Gaussian, we get
$$\mathrm{Ent}\left(S_n||G\right) = \mathrm{Ent}\left(\sqrt{\sigma}S_n||\sqrt{\sigma}G\right) \leq \frac{2\left(d + 2\mathrm{Ent}\left( \sqrt{\sigma}X||G\right)\right)}{\sigma^4n}.$$
\end{remark}
\subsection{An informal description of the method} \label{sec: method description}

Let $B_t$ be a standard Brownian motion in $\RR^d$ with an associated filtration $\FF_t$. The following definition will be central to our method:

\begin{definition}
Let $X_t$ be a martingale satisfying $dX_t = \Gamma_tdB_t$ for some adapted process $\Gamma_t$ taking values in the positive semi-definite cone and let $\tau$ be a stopping time. We say that the triplet $(X_t, \Gamma_t, \tau)$ is a martingale embedding of the measure  $\mu$ if $X_\tau \sim \mu$.
\end{definition}

Note that if $\Gamma_t$ is deterministic, then $X_t$ has a Gaussian law for each $t$. At the heart of our proof is the following simple idea: Summing up $n$ independent copies of a martingale embedding of $\mu$, we end up with a martingale embedding of $\mu^{* n}$ whose associated covariance process has the form $\sqrt{\sum_{i = 1}^n
\left(\Gamma^{(i)}_t\right)^2}$. By the law of large numbers, this process is well concentrated and thus the resulting martingale is close to a Brownian motion.

This suggests that it would be useful to couple the sum process $\sum_{i = 1}^n X^{(i)}_t$ with the "averaged" process whose covariance is given by $\EE\left[ \sqrt{\sum_{i = 1}^n
    \left(\Gamma^{(i)}_t\right)^2} \right]$ (this process is a Brownian
motion up to deterministic time change). Controlling the error in the coupling naturally leads to a
bound on transportation distance. For relative entropy, we can
reformulate the discrepancies in the coupling in terms of a
predictable drift and deduce bounds by a judicious application of
Girsanov's theorem.

In order to derive quantitative bounds, one needs to construct a martingale embedding in a way that makes the fluctuations of the process $\Gamma_t$ tractable. The specific choices of $\Gamma_t$ that we consider are based on a construction introduced in \cite{eldan2016skorokhod}.
This construction is also related to the entropy minimizing process used by F\"ollmer (\cite{follmer1985entropy,follmer1986time}, see also Lehec \cite{lehec2013representation}) and to the stochastic localization which was used in \cite{eldan2013thin}. Such techniques have recently gained prominence and have been used, among other things, to improve known bounds of the KLS conjecture \cite{eldan2013thin,lee2016eldan}, calculate large deviations of non-linear functions \cite{eldan2016gaussian} and study tubular neighborhoods of complex varieties \cite{klartag2017eldan}.

The basic idea underlying the construction of the martingale is a certain measure-valued Markov process driven by a Brownian motion. This process interpolates between a given measure and a delta measure via multiplication by infinitesimal linear functions. The Doob martingale associated to the delta measure (the conditional expectation of the measure, based on the past) will be a martingale embedding for the original measure. This construction is described in detail in Subsection \ref{construction} below.
\subsection{Related work}

Multidimensional central limit theorems have been studied extensively
since at least the 1940's \cite{bergstrom1945central} (see also \cite{bhattacharya1977refinements} and
references therein). In particular, the dependence of the convergence
rate on the dimension was studied by Nagaev \cite{nagaev1976estimate}, Senatov
\cite{senatov1981uniform}, G\"otze \cite{gotze1991rate}, Bentkus \cite{bentkus2005lyapunov}, and Chen and Fang
\cite{chen2011multivariate}, among others. These works focused on convergence in
probabilities of convex sets. We mention that in dimension 1, the picture is much clearer and that tight estimates are known under various metrics (\cite{berry1941accuracy,bobkov13approach,bobkov18bounds,esseen1942liapounoff,rio09upper,rio11asymtotic}).

More recently, dependence on dimension in the high-dimensional CLT has
also been studied for Wishart matrices (Bubeck and Ganguly
\cite{bubeck2016entropic}, Eldan and
Mikulincer \cite{eldan2016information}), maxima of sums of independent
random vectors (Chernozhukov, Chetverikov, and Kato
\cite{chernozhukov2013gaussian}), and transportation distance
(\cite{zhai2016multivariate}). As mentioned earlier, Theorem
\ref{thm:w2-bounded} is directly comparable to an earlier result of
the third named author \cite{zhai2016multivariate}, improving on it by
a factor of $\sqrt{\log n}$ (see also the earlier work
\cite{valiant2011estimating}). We refer to \cite{zhai2016multivariate}
for a discussion of how convergence in transportation distance may be
related to convergence in probabilities of convex sets.

As mentioned above, Theorem \ref{thm:w2-log-concave} is not new, and follows from a result of Courtade, Fathi and Pananjady \cite[Theorem 4.1]{courtade2017existence}. Their technique employs Stein's method (see also \cite{bonis2019stein}, for a different approach using Stein's method) in a novel way which is also applicable to entropic CLTs (see below). In a subsequent work \cite{fathi2018stein}, similar bounds are derived for convergence in the $p$'th-Wasserstein transportation metric.

Regarding entropic CLTs, it was shown by Barron
\cite{barron1986entropy} that convergence occurs as long as the
distribution of the summand has finite relative entropy (with respect
to the Gaussian). However, establishing explicit rates of convergence
does not seem to be a straightforward task. Even in the restricted
setting of log-concave distributions, not much is known.
One of the only quantitative results is Proposition 4.3 in \cite{courtade2017existence}, which gives near optimal convergence, provided that the distribution has finite Fisher information. We do not know of any results
prior to Theorem \ref{thm:ent-log-concave} which give entropy distance
bounds of the form $\frac{\mathrm{poly}(d)}{n}$ to a sum of general log-concave vectors.

A one-dimensional result was established by Artstein, Ball, Barthe,
and Naor \cite{artstein2004rate} and independently by Barron and Johnson
\cite{barron2004fisher}, who showed an optimal $O(1/n)$ convergence
rate in relative entropy for distributions having a spectral gap
(i.e. satisfying a Poincar\'e inequality). This was later improved by
Bobkov, Chistyakov, and G\"otze \cite{bobkov2013rate,
  bobkov2014berry}, who derive an Edgeworth-type expansion for the
entropy distance which also applies to higher dimensions. However,
although their estimates contain very precise information as $n
\rightarrow \infty$, the given error term is only asymptotic in $n$ and no explicit
dependence on the measure or on the dimension is given (in fact, the dependence
derived from the method seems to be exponential in the dimension
$d$).

A related ``entropy jump'' bound was proved by Ball and Nguyen
\cite{ball2012entropy} for log-concave random vectors in arbitrary
dimensions (see also \cite{ball2003entropy}). Essentially, the bound
states that for two i.i.d. random vectors $X$ and $Y$, the relative
entropy $\Ent\left(  \frac{X + Y}{\sqrt{2}} \middle|\middle|G \right)$
is strictly less than $\Ent(X || G)$, where the amount is quantified
by the spectral gap for the distribution of
$X$. Repeated application gives a bound for entropy of sums of
i.i.d. log-concave vectors in any dimension, but the bound is far from
optimal. It is not apparent to us whether the method of
\cite{ball2012entropy} can be extended to provide quantitative
estimates for convergence in the entropic CLT.

\subsection{Notation}
We work in $\RR^d$ equipped with the Euclidean norm, which we denote by $\norm{\cdot}$. For a positive semi-definite symmetric matrix $A$ we denote by $\sqrt{A}$ the unique positive semi-definite matrix $B$, for which the relation $B^2 = A$ holds. For symmetric matrices $A$ and $B$ we use $A \preceq B$ to signify that $B - A$ is a positive semi-definite matrix. By $A^\dagger$ we denote the pseudo inverse of $A$. Put succinctly, this means that in $A^\dagger$ every non-zero eigenvalue of $A$ is inverted. For a random matrix $A$, we will write $\EE\left[A\right]^\dagger$, for the pseudo inverse of its expectation.\\
If $B_t$ is the standard Brownian motion in $\RR^d$ then for any adapted process $F_t$ we denote by
$\int\limits_0^tF_sdB_s,$
the It\^o stochastic integral. We refer by It\^o's isometry to the fact
$$\EE\left[\norm{\int\limits_0^tF_sdB_s}^2\right] = \int\limits_0^t\EE\left[\norm{F_s}_{HS}^2\right]ds$$
when $F_t$ is adapted to the natural filtration of $B_t$.\\
$\mu$ will always stand for a probability measure. To avoid confusion, when integrating with respect to $\mu$, on $\mathbb{R}^d$, we will use the notation $\int \dots \mu(dx)$. For a measure-valued stochastic process $\mu_t$, the expression $d\mu_t$ refers to the stochastic derivative of the process.
A measure $\mu$ on $\RR^d$ is said to be log-concave if it is supported on some subspace of $\RR^d$ and, relative to the Lebesgue measure of that subspace, it has a density $\rho$, twice differentiable almost everywhere, for which
$$-\nabla^2\log(\rho(x)) \succeq 0 \ \ \ \ \text{for all } x,$$
where $\nabla^2$ denotes the Hessian matrix, in the Alexandrov sense. If in addition there exists an $\sigma > 0$ such that
$$-\nabla^2\log(\rho(x)) \succeq \sigma\mathrm{I}_d \ \ \ \ \text{for all } x,$$
we say that $\mu$ is $\sigma$-uniformly log-concave. The measure $\mu$ is called \emph{isotropic} if it is centered and its covariance matrix is the identity, i.e.,
 $$\int\limits_{\RR^d}x \mu(dx) = 0 \text{ and } \int\limits_{\RR^d}x\otimes x \mu(dx) = \mathrm{I}_d.$$
Finally, as a convention, we use the letters $C,C',c,c'$ to represent positive universal constants whose values may change between different appearances.

\paragraph{\bf Acknowledgments}  We are extremely grateful to the anonymous referee for his/her careful reading of this manuscript. His/her efforts have greatly improved the presentation and overall readability.
\section{Obtaining convergence rates from martingale embeddings}
Suppose that we are given a measure $\mu$ and a corresponding martingale embedding $(X_t, \Gamma_t, \tau)$. The goal of this section is to express bounds for the corresponding CLT convergence rates (of the sum of independent copies of $\mu$-distributed random vectors) in terms of the behavior of the process $\Gamma_t$ and $\tau$.

Throughout this section we fix a measure $\mu$ on $\RR^d$ whose expectation is $0$, a random vector $X \sim \mu$, and a corresponding Gaussian $G \sim \mathcal{N}\left(0, \Sigma\right)$, where $\mathrm{Cov}\left(X\right) = \Sigma$. Also, the sequence $\{X^{(i)}\}_{i=1}^\infty$ will denote independent copies of $X$, and we write $S_n := \frac{1}{\sqrt{n}}\sum\limits_{i=1}^nX^{(i)}$ for their normalized sum. Finally, we use $B_t$ to denote a standard Brownian motion on $\RR^d$ adapted to a filtration $\FF_t$.

\subsection{A bound for Wasserstein-2 distance}
The following is our main bound for convergence in Wasserstein distance.
\begin{theorem} \label{main}
	Let $S_n$ and $G$ be defined as above and let $(X_t, \Gamma_t, \tau)$ be a martingale embedding of $\mu$. Set $\Gamma_t = 0$ for $t > \tau$, then
	$$\mathcal{W}_2^2\left(S_n, G\right) \leq \int\limits_{0}^\infty \min\left(\frac{1}{n}\mathrm{Tr}\left(\EE\left[\Gamma_t^4\right]\EE[\Gamma_t^2]^{\dagger}\right), 4\mathrm{Tr}\left(\EE\left[\Gamma_t^2\right]\right)\right)dt.$$
\end{theorem}

To illustrate how such a result might be used, let us assume, for simplicity, that $\Gamma_t \prec k\mathrm{I}_d$ almost-surely for some $k > 0 $ and that $\tau$ has a sub-exponential tail, i.e., there exist positive constants $C, c > 0$ such that for any $t > 0$,
\begin{equation} \label{sub-exponential}
\PP(\tau > t) \leq C e^{-ct}.
\end{equation}
Under these assumptions,
\begin{align*}
\mathcal{W}_2^2\left(S_n, G\right) &\leq \int\limits_{0}^\infty \min\left(\frac{1}{n}\mathrm{Tr}\left(\EE\left[\Gamma_t^4\right]\EE[\Gamma_t^2]^{\dagger}\right), 4k^2d\PP\left(\tau  > t\right)\right)dt \\
&\leq dk^2\int\limits_0^{\frac{\log(n)}{c}}\frac{1}{n}dt + 4Cdk^2\int\limits_{\frac{\log(n)}{c}}^\infty e^{-ct}dt = \frac{d\log(n)k^2}{cn} + \frac{4C dk^2}{n}.
\end{align*}
Towards the proof, we will need the following technical lemma.
\begin{lemma} \label{PositiveDefinite}
	Let $A,B$ be positive semi-definite matrices with $\ker(A) \subset \ker(B)$. Then,
	$$\mathrm{Tr}\left(\left(\sqrt{A} - \sqrt{B}\right)^2\right) \leq \mathrm{Tr}\left(\left(A-B\right)^2A^{\dagger}\right).$$
\end{lemma}
\begin{proof}
	Since $A$ and $B$ are positive semi-definite, $\ker\left(\sqrt{A} + \sqrt{B}\right) \subset\ker\left(\sqrt{A} - \sqrt{B}\right)$.Thus, we have that
	\begin{align}
	\sqrt{A}-\sqrt{B} &= \left(\sqrt{A}-\sqrt{B}\right)\left(\sqrt{A}+\sqrt{B}\right)\left(\sqrt{A}+\sqrt{B}\right)^{\dagger}\\
	&= \left(A  - B +\left[\sqrt{A},\sqrt{B}\right]\right)\left(\sqrt{A}+\sqrt{B}\right)^{\dagger}.\nonumber
	\end{align}
	So, $$\mathrm{Tr}\left(\left(\sqrt{A}-\sqrt{B}\right)^2\right) = \mathrm{Tr}\left(\left(\left(A  - B +\left[\sqrt{A},\sqrt{B}\right]\right)\left(\sqrt{A}+\sqrt{B}\right)^{\dagger}\right)^2\right).$$
	Note that for any symmetric matrices $X$ and $Y$, by the Cauchy-Schwartz inequality,
	\[ \mathrm{Tr}\left((XY)^2\right) \leq \mathrm{Tr}\left(XYXY\right)\leq \sqrt{\mathrm{Tr}\left(XYYX\right)\cdot\mathrm{Tr}\left(YXXY\right)}= \mathrm{Tr}\left(X^2Y^2\right). \]
	Applying this to the above equation shows
	\begin{align*}
	\mathrm{Tr}\left(\left(\sqrt{A}-\sqrt{B}\right)^2\right) \leq \mathrm{Tr}\left(\left(A - B + \left[\sqrt{A},\sqrt{B}\right]\right)^2\left(\left(\sqrt{A}+\sqrt{B}\right)^\dagger\right)^{2}\right).
	\end{align*}

	Note that the commutator $\left[\sqrt{A},\sqrt{B}\right]$ is an anti-symmetric matrix, so that $(A-B)\left[\sqrt{A},\sqrt{B}\right] + \left[\sqrt{A},\sqrt{B}\right](A-B)$ is anti-symmetric as well. Thus, for any symmetric matrix $C$, we have that
	$$\mathrm{Tr}\left(\left((A-B)\left[\sqrt{A},\sqrt{B}\right] + \left[\sqrt{A},\sqrt{B}\right](A-B)\right)C\right) = 0.$$ Also, since all eigenvalues of anti-symmetric matrices are purely imaginary, the square of such matrices must be negative definite. And again, for any symmetric positive definite matrix $C$, it holds that $C^{1/2}\left[\sqrt{A},\sqrt{B}\right]^2C^{1/2}$ is negative definite and $\mathrm{Tr}\left(\left[\sqrt{A},\sqrt{B}\right]^2C\right) \leq 0$. Using these observations we obtain
	$$\mathrm{Tr}\left(\left(A - B + \left[\sqrt{A},\sqrt{B}\right]\right)^2\left(\left(\sqrt{A}+\sqrt{B}\right)^\dagger\right)^{2}\right) \leq \mathrm{Tr}\left(\left(A-B\right)^2\left(\left(\sqrt{A}+\sqrt{B}\right)^\dagger\right)^{2}\right).$$
	Finally, if $C,X,Y$ are positive definite matrices with $X \preceq Y$ then $C^{1/2}(Y-X)C^{1/2}$ is positive definite which shows $\mathrm{Tr}\left(CX\right) \leq \mathrm{Tr}\left(CY\right)$. The assumption $\ker(A) \subset \ker(B)$ implies $\left(\left(\sqrt{A}+\sqrt{B}\right)^{\dagger}\right)^2\preceq A^{\dagger}$, which concludes the claim by
	$$\mathrm{Tr}\left(\left(A-B\right)^2\left(\left(\sqrt{A}+\sqrt{B}\right)^{\dagger}\right)^{2}\right) \leq \mathrm{Tr}\left(\left(A-B\right)^2A^{\dagger}\right) $$
\end{proof}
\begin{proof}[Proof of Theorem \ref{main}]
	Recall that $(X_t, \Gamma_t, \tau)$ is a martingale embedding of $\mu$. Let $\left(X_t^{(i)}, \Gamma_t^{(i)}, \tau^{(i)}\right)$ be independent copies of the embedding. We can always set $\Gamma_t^{(i)} = 0$ whenever $t > \tau^{(i)} $, so that $\int\limits_{0}^\infty \Gamma^{(i)} _tdB^{(i)} _t \sim \mu$. Define $\tilde{\Gamma}_t = \sqrt{\frac{1}{n}\sum\limits_{i=1}^n\left(\Gamma_t^{(i)} \right)^2}$. Our first goal is to show
	\begin{equation} \label{wassersein-bound}
	\mathcal{W}_2^2(G,S_n) \leq  \int\limits_0^\infty\EE\left[\mathrm{Tr}\left(\left(\tilde{\Gamma_t} - \sqrt{\EE\left[\Gamma_t^2\right]}\right)^2\right)\right]dt.
	\end{equation}
	The theorem will then follow by deriving suitable bounds for $\EE\left[\mathrm{Tr}\left(\left(\tilde{\Gamma_t} - \sqrt{\EE\left[\Gamma_t^2\right]}\right)^2\right)\right]$ using Lemma \ref{PositiveDefinite}.
	Consider the sum
	$\frac{1}{\sqrt{n}}\sum\limits_{i=1}^n \int\limits_{0}^\infty \Gamma^{(i)}_tdB^{(i)}_t$, which has the same law as $S_n$. It may be rewritten as
	$$S_n = \int\limits_{0}^\infty \tilde{\Gamma}_td\tilde{B}_t,$$
	where $d \tilde{B}_t :=  \frac{1}{\sqrt{n}} \tilde{\Gamma}_t^\dagger \sum_i \Gamma_t^{(i)} d B_t^{(i)}$
	is a martingale whose quadratic variation matrix has derivative satisfying
	\begin{equation}\label{eq:qvid}
	\frac{d}{dt} [\tilde B]_t = \frac{1}{n }\sum_i \tilde{\Gamma}_t^\dagger \left ( \Gamma_t^{(i)} \right )^2 \tilde{\Gamma}_t^\dagger \preceq \mathrm{I}_d.
	\end{equation}
	(in fact, as long as $\RR^d$ is spanned by the images of $\Gamma_t^{(i)}$, this process is a Brownian motion). We may now decompose $S_n$ as
	\begin{equation} \label{decompo}
	S_n = \int\limits_{0}^\infty\sqrt{\EE\left[\tilde{\Gamma}_t^2\right]}d\tilde{B}_t + \int\limits_{0}^\infty\left(\tilde{\Gamma}_t - \sqrt{\EE\left[\tilde{\Gamma}_t^2\right]}\right)d\tilde{B}_t.
	\end{equation}
	Observe that $G := \int\limits_{0}^\infty \sqrt{\EE[\tilde{\Gamma}_t^2]}d\tilde{B}_t$ has a Gaussian law and that $\EE[\tilde{\Gamma}_t^2] = \EE[\Gamma_t^2]$. By applying It\^o's isometry, we may see that $G$ has the ``correct'' covariance in the sense that
	$$\mathrm{Cov}(G) = \EE\left[\left(\int\limits_{0}^\infty \sqrt{\EE[\tilde{\Gamma}_t^2]}d\tilde{B}_t\right)^{\otimes 2}\right] = \EE\left[\int\limits_{0}^\infty \Gamma^2_tdt\right] = \EE\left[\left(\int\limits_{0}^\infty \Gamma_tdB_t\right)^{\otimes 2}\right] = \mathrm{Cov}(X).$$
	The decomposition \eqref{decompo} induces a natural coupling between $G$ and $S_n$, which shows, by another application of It\^o's isometry, that
	\begin{align*}
	\mathcal{W}_2^2(G,S_n) &\leq \EE\left[\norm{\int\limits_0^\infty \left( \tilde{\Gamma}_t - \sqrt{\EE[\Gamma_t^2]} \right) d\tilde{B}_t}^{2}\right] \stackrel{ \eqref{eq:qvid}}{ \leq}  \mathrm{Tr}\left(\EE\left[\int\limits_{0}^\infty\left(\tilde{\Gamma}_t - \sqrt{\EE[\Gamma_t^2]}\right)^2dt\right]\right)\\
	& = \int\limits_0^\infty\EE\left[\mathrm{Tr}\left(\left(\tilde{\Gamma_t} - \sqrt{\EE\left[\Gamma_t^2\right]}\right)^2\right)\right]dt ,\nonumber
	\end{align*}
	where the last equality is due to Fubini's theorem. Thus, \eqref{wassersein-bound} is established. Since $\left(\tilde{\Gamma_t} - \sqrt{\EE\left[\Gamma_t^2\right]}\right)^2 \preceq 2\left(\tilde{\Gamma}_t^2 + \EE[\Gamma^2_t]\right)$, we have
	\begin{equation} \label{first half}
	\mathrm{Tr}\left(\EE\left[\left(\tilde{\Gamma_t} - \sqrt{\EE\left[\Gamma_t^2\right]}\right)^2\right]\right)\leq 4\mathrm{Tr}\left(\EE[\Gamma_t^2]\right).
	\end{equation}
	To finish the proof, write  $U_t := \frac{1}{n}\sum\limits_{i = 1}^n\left(\Gamma_t^{(i)}\right)^2$, so that $\tilde{\Gamma}_t = \sqrt{U_t}$. Since $\Gamma_t$ is positive semi-definite, it is clear that $ \ker\left(\EE\left[\Gamma_t^2\right]\right) \subset \ker(U_t)$. By Lemma \ref{PositiveDefinite},
	\begin{align*}
	\EE\left[\mathrm{Tr}\left(\left(\sqrt{U_t} - \sqrt{\EE\left[\Gamma_t^2\right]}\right)^2\right)\right] \leq& \mathrm{Tr}\left(\EE\left[\left(U_t - \EE\left[\Gamma_t^2\right]\right)^2\right]\EE\left[\Gamma_t^2\right]^{\dagger}\right) \\
	=& \frac{1}{n^2}\mathrm{Tr}\left(\sum\limits_{i=1}^n\EE\left[\left(\left(\Gamma_t^{(i)}\right)^2 - \EE\left[\Gamma_t^2\right]\right)^2\right]\EE\left[\Gamma_t^2\right]^{\dagger}\right)\\
	=& \frac{1}{n}\mathrm{Tr}\left(\left(\EE\left[\Gamma_t^4\right] - \EE\left[\Gamma_t^2\right]^2\right)\EE\left[\Gamma_t^2\right]^{\dagger}\right)\\
	 \leq& \frac{1}{n}\mathrm{Tr}\left(\EE\left[\Gamma_t^4\right]\EE\left[\Gamma_t^2\right]^{\dagger}\right),
	\end{align*}
	where we have used the fact $\EE\left[\left(\Gamma_t^{(i)}\right)^2\right] = \EE\left[\Gamma_t^2\right]$ in the second equality.
	Combining  the last inequality with \eqref{first half} and \eqref{wassersein-bound} produces the required result.

\end{proof}
\subsection{A bound for the relative entropy}
As alluded to in the introduction, in order to establish bounds on the relative entropy we will use the existence of a martingale embedding to construct an It\^o process whose martingale part has a deterministic quadratic variation. This will allow us to relate the relative entropy to a Gaussian with the norm of the drift term through the use of Girsanov's theorem.  As a technicality, we require the stopping time associated to the martingale embedding to be constant. Our main bound for the relative entropy reads,
\begin{theorem} \label {thm: quant entropic clt}
Let $(X_t,\Gamma_t,1)$ be a martingale embedding of $\mu$. Assume that for every $0 \leq t \leq 1$, $\EE\left[\Gamma_t\right] \succeq \sigma_t\mathrm{I}_d \succneqq 0$ and that $\Gamma_t$ is invertible a.s. for $t < 1$. Then we have the following inequalities:
$$\mathrm{Ent}(S_n|| G) \leq \frac{1}{n}\int\limits_0^1\frac{\EE\left[\mathrm{Tr}\left(\left(\Gamma_t^2 - \EE\left[\Gamma_t^2\right]\right)^2\right)\right]}{(1-t)^2\sigma_t^2}\left(\int\limits_t^1\sigma_s^{-2}ds\right)dt,$$
and
$$\mathrm{Ent}(S_n|| G) \leq \int\limits_0^1\frac{\mathrm{Tr}\left(\EE\left[\Gamma_t^2 \right] - \EE\left[\tilde{\Gamma_t}\right]^2\right)}{(1-t)^2}\left(\int\limits_t^1\sigma_s^{-2}ds\right)dt,$$
where
$$\tilde{\Gamma}_t = \sqrt{\frac{1}{n}\sum\limits_{i=1}^n \left(\Gamma_t^{(i)}\right)^2}$$
 and $\Gamma_t^{(i)}$ are independent copies of $\Gamma_t$.
\end{theorem}
The theorem relies on the following bound, whose proof is postponed to the end of the subsection.
\begin{lemma} \label{lem:rel entropy method}
	Let $\Gamma_t$ be an $\FF_t$-adapted matrix-valued processes and let $F:\RR\times\RR^d \to \RR^{d\times d}$ be almost surely invertible and locally Lipschitz. Denote $F_t(x) := F(t,x)$ and let $X_t, M_t$ be defined by
	$$
	X_t = \int_0^t \Gamma_s dB_s ~~ \mbox{and} ~~ M_t = \int_0^t F_s(M_s) dB_s.
	$$
	Define the process $Y_t$ by
	$$Y_t = \int\limits_0^tF_s(Y_s)dB_s + \int\limits_{0}^t\int\limits_{0}^s\frac{\Gamma_r - F_r(Y_r)}{1-r}dB_rds.$$
	Then,
	$$\mathrm{Ent}\left(X_1||M_1\right) \leq \EE\left[\int\limits_0^1\int\limits_s^1 \norm{F_t^{-1}(Y_t)\frac{\Gamma_s - F_s(Y_s)}{1-s}}_{HS}^2dtds\right]. $$
\end{lemma}
Note that if the process $F_t$ is deterministic, i.e. it is a constant function, then $M_1$ has a Gaussian law, so that the lemma can be used to bound the relative entropy of $X_1$ with respect to a Gaussian.

\begin{proof}[Proof of Theorem \ref{thm: quant entropic clt}]
	Let $(X^{(i)}_t,\Gamma_t^{(i)},1)$ be independent copies of the martingale embedding. Consider the sum process $\tilde{X}_t = \frac{1}{\sqrt{n}} \sum\limits_{i=1}^nX_t^{(i)}$, which satisfies $\tilde{X}_t = \int\limits_0^t\tilde{\Gamma}_sd\tilde{B}_s$ where we define, as in the proof of Theorem \ref{main},
	$$\tilde{\Gamma_t} := \sqrt{\frac{1}{n}\sum\limits_{i=1}^n\left(\Gamma_t^{(i)}\right)^2} \text{ and } d\tilde{B}_t = \frac{1}{\sqrt{n}} \tilde{\Gamma}_t^{-1} \sum \Gamma_t^{(i)}dB_t^{(i)}.$$
	By assumption $\tilde{\Gamma}_t$ is invertible, which makes $\tilde{B}_t$ a Brownian motion. In this case, $(\tilde{X}_t, \tilde{\Gamma}_t, 1)$ is a martingale embedding for the law of $S_n$. For the first bound, consider the process
	$$
	M_t = \int_0^t \sqrt{\EE\left[\Gamma_s^2\right]}d\tilde{B}_s.
	$$
	By It\^o's isometry one has $M_1 \sim \mathcal{N}\left(0,\Sigma\right)$. Also, by Jensen's inequality
	$$\sqrt{\EE\left[\Gamma_t^2\right]} \succeq \EE\left[\Gamma_t\right] \succeq \sigma_t\mathrm{I}_d.$$
	Using this observation and substituting $\sqrt{\EE\left[\Gamma_t^2\right]}$ for a constant function $F_t$ in Lemma \ref{lem:rel entropy method} yields,
	\begin{equation} \label{cor equation}
	\mathrm{Ent}\left(S_n \| G\right)
	\leq \int\limits_0^1\EE\left[\norm{\frac{\tilde{\Gamma_t} - \sqrt{\EE\left[\Gamma_t^2\right]}}{1-t}}_{HS}^2\right]\left(\int\limits_t^1\sigma_s^{-2}ds\right)dt.
	\end{equation}
	With the use of Lemma \ref{PositiveDefinite} we obtain
	\begin{align*}
	\EE\norm{\tilde{\Gamma_t}-\sqrt{\EE\left[\Gamma_t^2\right]}}_{HS}^2 &= \EE\left[\mathrm{Tr}\left(\left(\tilde{\Gamma_t}-\sqrt{\EE\left[\Gamma_t^2\right]}\right)^2\right)\right]\\
	&\leq \EE\left[\mathrm{Tr}\left(\left(\frac{1}{n}\sum\limits_{i=1}^n\left(\Gamma_t^{(i)}\right)^2 - \EE\left[\Gamma_t^2\right]\right)^2\EE\left[\Gamma_t^2\right]^{-1}\right)\right]\\
	 &\leq \frac{1}{n\sigma^2_t}\EE\left[\mathrm{Tr}\left(\left(\Gamma_t^2 - \EE\left[\Gamma_t^2\right]\right)^2\right)\right].
	\end{align*}
	Plugging the above into \eqref{cor equation} shows the first bound.
	To see the second bound, we define a process $M'_t$, which is similar to $M_t$, and is given by the equation
	$$
	M'_t := \int_0^t \EE\left[\tilde{\Gamma_s}\right]d\tilde{B}_s.
	$$
	Let $G_n$ denote a Gaussian which is distributed as $M'_1$. For any $s$, we now have the following Cauchy-Schwartz type inequality
	$$n\left(\sum_{i=1}^n\left(\Gamma_s^{(i)}\right)^2\right)\succeq \left(\sum_{i=1}^n\Gamma^{(i)}_s\right)^2.$$
	Since the square root is monotone with respect to the order on positive definite matrices, this implies $$\EE\left[\tilde{\Gamma_s}\right] \succeq  \frac{1}{n}\EE\left[\sum_{i=1}^n\Gamma_s^{(i)}\right]\succeq \sigma_s\mathrm{I}_d.$$
	Thus,
	\begin{align*}
	\mathrm{Ent}( S_n||G_n)& \leq \EE\left[\int\limits_0^1\int\limits_t^1 \norm{\EE\left[\tilde{\Gamma}_s\right]^{-1}\frac{\tilde{\Gamma_t} - \EE\left[\tilde{\Gamma}_t\right]}{1-t}}_{HS}^2dsdt\right]\\
	&\leq \int\limits_0^1 \EE\left[\norm{\frac{\tilde{\Gamma_t} - \EE\left[\tilde{\Gamma}_t\right]}{1-t}}_{HS}^2\right]\left(\int\limits_t^1\sigma_s^{-2}ds\right)dt\\
	& =
	\int\limits_0^1 \frac{\mathrm{Tr}\left(\EE\left[\Gamma_t^2\right] - \EE\left[\tilde{\Gamma}_t\right]^2\right)}{(1-t)^2}\left(\int\limits_t^1\sigma_s^{-2}ds\right)dt.
	\end{align*}
	Since $\mathrm{Cov}(G) = \mathrm{Cov}(S_n),$ it is now easy to verify that $\mathrm{Ent}\left(S_n||G\right) \leq \mathrm{Ent}\left(S_n||G_n\right)$, which concludes the proof.
\end{proof}
A key component in the proof of the theorem lies in using the norm of an adapted process in order to bound the relative entropy. The following lemma embodies this idea. Its proof is based on a straightforward application of Girsanov's theorem. We provide a sketch and refer the reader to \cite{lehec2013representation}, where a slightly less general version of this lemma is given, for a more detailed proof.
\begin{lemma} \label{lem:entropy bound}
	Let $F:\RR \times \RR^d \to \RR^{d\times d}$ be almost surely invertible and locally Lipschitz. Denote $F_t(x) := F(t,x)$ and let $M_t = \int\limits_0^tF_s(M_s)dB_s$. For $u_t$, an adapted process, set $Y_t := \int\limits_0^tF_s(Y_s)dB_s + \int\limits_0^t u_s ds$. Then
	$$\mathrm{Ent}\left(Y_1||M_1\right) \leq \frac{1}{2}\int\limits_0^1\EE\left[\norm{F_t^{-1}(Y_t)u_t}^2\right]dt.$$
\end{lemma}
\begin{proof}
	Since $M_t$ is an It\^o diffusion, by Girsanov's theorem (\cite[Theorem 8.6.5]{oksendal2003stochastic}), the density of $\{Y_t\}_{t \in [0,1]}$ with respect to that of $\{ M_t \}_{t \in [0,1]}$ on the space of paths is given by
	$$\mathcal{E} := \exp\left(-\int\limits_0^1F_t(Y_t)^{-1}u_tdB_t - \frac{1}{2}\int\limits_0^1\norm{F_t(Y_t)^{-1}u_t}^2dt\right).$$
	If $f$ is the density of $Y_1$ with respect to $M_1$, this implies
	$$1 = \EE\left[f(Y_1)\mathcal{E}\right].$$
	By Jensen's inequality
	$$0 = \ln\left(\EE\left[f(Y_1)\mathcal{E}\right]\right) \geq \EE\left[\ln\left(f(Y_1)\mathcal{E}\right)\right] = \EE\left[\ln(f(Y_1))\right] + \EE\left[\ln(\mathcal{E})\right].$$
	But,
	$$\EE\left[\ln(\mathcal{E})\right] = - \frac{1}{2}\int\limits_0^1\EE\left[\norm{F_t^{-1}(Y_t)u_t}^2\right]dt,$$
	and
	$$\EE\left[\ln(f(Y_1))\right] = \mathrm{Ent}(Y_1||M_1),$$
	which concludes the proof.
\end{proof}

The proof of Lemma \ref{lem:rel entropy method} now amounts to invoking the above bound with a suitable construction of the drift process $u_t$.

\begin{proof}[Proof of Lemma \ref{lem:rel entropy method}]
By defintion of the process $Y_t$, we have the following equality
\begin{equation} \label{fubini}
Y_1 = \int\limits_0^1F_t(Y_t)dB_t + \int\limits_0^1\int\limits_0^t\frac{\Gamma_s-F_s(Y_s)}{1-s}dB_sdt = \int\limits_0^1F_t(Y_t)dB_t + \int\limits_0^1\left(\Gamma_t-F_t(Y_t)\right)dB_t =X_1,
\end{equation}
where we have used Fubini's theorem in the penultimate equality.
Now, consider the adapted process
$$u_t = \int\limits_0^t \frac{\Gamma_s - F_s(Y_s)}{1-s}dB_s,$$
so that,
$$dY_t = F_t(Y_t)dB_t + u_tdt.$$
Applying Lemma \ref{lem:entropy bound} and using It\^o's isometry, we get
\begin{align*}
\mathrm{Ent}( X_1||M_1)&\leq\int\limits_0^1\EE\left[\norm{F_t^{-1}(Y_t)u_t}^2\right]dt
= \int\limits_0^1\EE\left[\norm{ \int\limits_0^tF_t^{-1}(Y_t)\frac{\Gamma_s - F_s(Y_s)}{1-s}dB_s}^2\right]dt \\
&= \EE\left[\int\limits_0^1\int\limits_0^t\norm{F_t^{-1}(Y_t)\frac{\Gamma_s - F_s(Y_s)}{1-s}}_{HS}^2dsdt\right] \\
& = \EE\left[\int\limits_0^1\int\limits_s^1 \norm{F_t(Y_t)^{-1}\frac{\Gamma_s - F_s(Y_s)}{1-s}}_{HS}^2dtds\right],
\end{align*}
where last equality follows from another use of Fubini's theorem.
\end{proof}

\subsection{A stochastic construction} \label{construction}

In this section we introduce the main construction used in our proofs,
a martingale process which meets the assumptions of Theorems
\ref{main} and \ref{thm: quant entropic clt}. The construction in the next proposition is based on the
Skorokhod embedding described in \cite{eldan2016skorokhod}. Most of
the calculations in this subsection are very similar to what is done
in \cite{eldan2016skorokhod}, except that we allow some inhomogeneity
in the quadratic variation according to the function $C_t$
below. In particular, $C_t$ will be a symmetric matrix almost surely, and we will denote the space of $d\times d$ \emph{symmetric} matrices by $\mathrm{Sym}_d$.

\begin{proposition} \label{prop:stochastic-localization}
  Let $\mu$ be a probability measure on $\RR^d$ with smooth density and bounded
  support. For a probability measure-valued process $\mu_t$, let
  \[ a_t = \int_{\RR^d} x \mu_t(dx), \quad A_t = \int_{\RR^d} (x - a_t)^{\otimes 2} \mu_t(dx) \]
  denote its mean and covariance.

  Let $C : \RR \times \mathrm{Sym}_d \to \mathrm{Sym}_d$ be a continuous
  function. Then, we can construct $\mu_t$ so that the following
  properties hold:
  \begin{enumerate}
  \item $\mu_0 = \mu$,
  \item $a_t$ is a stochastic process satisfying $da_t =
    A_tC(t, A_t^{\dagger})dB_t$, where $B_t$ is a standard Brownian
    motion on $\RR^d$, and \label{item:da_t}
  \item For any continuous and bounded $\varphi : \RR^d \to \RR$, $\int_{\RR^d} \varphi(x) \mu_t(dx)$ is a martingale. \label{item:martingale}
  \end{enumerate}
\end{proposition}

\begin{remark}
  We will be mainly interested in situations where $\mu_t$ converges
  almost surely to a point mass in finite time. In this case, we
  obtain a martingale embedding $(a_t, A_tC(t, A_t^{\dagger}), \tau)$
  for $\mu$, where $\tau$ is the first time that $\mu_t$ becomes a
  point mass.
\end{remark}

In the sequel, we abbreviate $C_t := C(t, A_t^\dagger)$. We first give an informal description of how $\mu_{t + \epsilon}$ is
constructed from $\mu_t$ for $\epsilon \rightarrow 0$. Consider a
stochastic process $\{ X_s \}_{0 \le s \le 1}$ in which we first
sample $X_1 \sim \mu_t$ and then set
\[ X_s = (1 - s) a_t + sX_1 + C_t^{-1} B_s, \]
where $B_s$ is a standard Brownian bridge. We can write $X_{\epsilon}
= a_t + \sqrt{\epsilon} C_t^{-1} Z$, where $Z$ is close
to a standard Gaussian. We then take $\mu_{t + \epsilon}$ to be the
conditional distribution of $X_1$ given $X_\epsilon$. This immediately
ensures that property \ref{item:martingale} holds and that $a_t$ is a
martingale.

It remains to see why property \ref{item:da_t} holds. A direct
calculation with conditioned Brownian bridges gives a first-order
approximation
\begin{align*}
  \mu_{t + \epsilon}(dx) &\propto e^{- \frac{1}{2} (\sqrt{\epsilon} C_t^{-1} Z - \epsilon(x - a_t))^T C_t^2 (\sqrt{\epsilon} C_t^{-1} Z - \epsilon(x - a_t))} \mu_t(dx) \\
  &\propto e^{\sqrt{\epsilon} \langle C_t Z, x - a_t \rangle + O(\epsilon)} \mu_t(dx) \\
  &\approx (1 + \sqrt{\epsilon} \langle C_t Z, x - a_t \rangle) \mu_t(dx).
\end{align*}
Then, to highest order, we have
\[ a_{t + \epsilon} - a_t \approx \sqrt{\epsilon} \int_{\RR^d} \langle C_t Z, x - a_t \rangle (x - a_t) \,\mu_t(dx) = \sqrt{\epsilon} A_t C_t Z, \]
which translates into property \ref{item:da_t} as $\epsilon
\rightarrow 0$.

Observe that the procedure outlined above yields measures $\mu_t$ that
have densities which are proportional to the original density $\mu$
times a Gaussian density. (This applies at least when $A_t$ is
non-degenerate; something similar also holds when $A_t$ is degenerate,
as we will see shortly.) Let us now perform the construction
formally. We will proceed by iterating the following preliminary
construction, which handles the case when $A_t$ remains
non-degenerate.

\begin{lemma} \label{lem:stochastic-localization-prelim}
  Let $\mu$ be a measure on $\RR^d$ with smooth density and bounded
  support, and let $C : \RR \times\mathrm{Sym}_d \to \mathrm{Sym}_d$ be a
  continuous map. Then, there is a measure-valued process $\mu_t$ and
  a stopping time $T$ such that $\mu_t$ satisfies the properties in
  Proposition \ref{prop:stochastic-localization} for $t < T$ and the
  affine hull of the support of $\mu_T$ has dimension strictly less
  than $d$. Moreover, if $\mu_T$ is considered as a measure on this
  affine hull, it has a smooth density.
\end{lemma}
\begin{proof}
  We will construct a $(\RR^d \times \mathrm{Sym}_d)$-valued
  stochastic process $(c_t, \tilde{\Sigma}_t)$ started at $(c_0, \tilde{\Sigma}_0) =
  (0, \Id)$. Let us write
  \[ Q_t(x) = \frac{1}{2} \left\langle x - c_t, \tilde{\Sigma}^{-1}_t (x - c_t) \right\rangle, \]
  and let $\tilde{\mu}$ be the probability measure satisfying $\frac{d
    \tilde{\mu}}{d\mu}(x) \propto e^{\frac{1}{2}\|x\|^2}$. We will
  then take $\mu_t$ to be $\mu_t(dx) = F_t(x) \tilde{\mu}(dx)$, where
  \[ F_t(x) = \frac{1}{Z_t} e^{-Q_t(x)}, \qquad Z_t = \int_{\RR^d} e^{-Q_t(x)} \tilde{\mu}(dx). \]
  Note that since $\tilde{\Sigma}_0 = \Id$, we have $\mu_0 =
  \mu$.\footnote{Conceptually, one can replace all instances of
    $\tilde{\mu}$ with $\mu$ if we think of the initial value
    $\tilde{\Sigma}_0$ as being an ``infinite'' multiple of
    identity. However, to avoid issues with infinities, we have
    expressed things in terms of $\tilde{\mu}$ instead.}

  In order to specify the process, it remains to construct $(c_t,
  \tilde{\Sigma}_t)$. We take it to be the solution to the SDE
  \[ dc_t = \tilde{\Sigma}_t C_t dB_t + \tilde{\Sigma}_t C_t^2 (a_t - c_t) dt, \qquad d\tilde{\Sigma}_t = -\tilde{\Sigma}_t C_t^2 \tilde{\Sigma}_t dt. \]
  Note that the coefficients of this SDE are continuous functions of
  $(c_t, \tilde{\Sigma}_t)$ so long as $\tilde{\Sigma}_t \succ 0$. By standard
  existence and uniqueness results, this SDE has a unique solution up
  to a stopping time $T$ (possibly $T = \infty$), at which point $A_t$
  (and hence $\tilde{\Sigma}_t$) becomes degenerate. Observe that, for every $t$, $\tilde{\Sigma}_t \preceq \mathrm{I}_d$ and so, the matrix process is continuous on the interval $[0,T]$. 

  By a limiting procedure, it is easy to see that $\mu_T$ has a smooth
  density when considered as a measure on the affine hull of its
  support. (Indeed, its density is proportional to the conditional
  density of $\tilde{\mu}$ times a Gaussian density.) It remains to
  verify that $\mu_t$ is a martingale and $da_t = A_t C_t
  dB_t$.

  By direct calculation, we have
  \begin{align*}
    d(\tilde{\Sigma}^{-1}_t) &= C_t^2 dt \\
    d(\tilde{\Sigma}^{-1}_t c_t) &= C_t^2 c_t dt + C_t^2(a_t - c_t) dt + C_t dB_t \\
    &= C_t^2 a_t dt + C_t dB_t \\
    dQ_t(x) &= \left\langle x, \left( \frac{1}{2} C_t^2 x - C_t^2 a_t \right) dt - C_t dB_t \right\rangle \\
    d(e^{-Q_t(x)}) &= -e^{-Q_t(x)} dQ_t(x) + \frac{1}{2} e^{-Q_t(x)} d[Q_t(x)] \\
    &= e^{-Q_t(x)} \left\langle x, C_t dB_t + C_t^2 a_t dt \right\rangle \\
  \end{align*}
  Integrating against $\tilde{\mu}(dx)$, we obtain
  \begin{align*}
    dZ_t &= Z_t \left\langle a_t, C_t dB_t + C_t^2 a_t dt \right\rangle \\
    d Z^{-1}_t &= -\frac{1}{Z_t^2} dZ_t + \frac{1}{Z_t^3} d[Z_t] = \frac{1}{Z_t} \langle a_t, -C_t dB_t \rangle \\
    dF_t(x) &= e^{-Q_t(x)} d Z^{-1}_t + Z^{-1}_t d(e^{-Q_t(x)}) + d[Z^{-1}_t, e^{-Q_t(x)}] \\
    &= F_t(x) \cdot \langle x - a_t, C_t dB_t \rangle.
  \end{align*}
  Thus, $F_t(x)$ is a martingale for each fixed $x$, and furthermore,
  \[ da_t = d \int_{\RR^d} x \mu_t(dx) = \int_{\RR^d} x d\mu_t(dx) = \int_{\RR^d} x(x - a_t)C_t\mu_t(dx) dB_t = A_t C_t dB_t. \]
\end{proof}

\begin{proof}[Proof of Proposition \ref{prop:stochastic-localization}]
  We use the process given by Lemma
  \ref{lem:stochastic-localization-prelim}, which yields a stopping
  time $T_1$ and a measure $\mu_{T_1}$ with a strictly
  lower-dimensional support. If $\mu_T$ is a point mass, then we set
  $\mu_t = \mu_T$ for all $t \ge T$.

  Otherwise, by the smoothness properties of $\mu_{T_1}$ guaranteed by
  Lemma \ref{lem:stochastic-localization-prelim}, we can recursively
  apply Lemma \ref{lem:stochastic-localization-prelim} again on
  $\mu_{T_1}$ conditioned on the affine hull of its support. Repeating
  this procedure at most $d$ times gives us the desired process.
\end{proof}

\subsection{Properties of the construction}

We record here various formulas pertaining to the quantities $a_t$,
$A_t$, and $\mu_t$ constructed in Proposition
\ref{prop:stochastic-localization}.

\begin{proposition} \label{prop:sigma-evolution}
  Let $\mu$, $C_t$, and $\mu_t$ be as in Proposition
  \ref{prop:stochastic-localization}. Then, there is a $\mathrm{Sym}_d$-valued process $\{ \Sigma_t \}_{t > 0}$ satisfying the
  following:
  \begin{itemize}
  \item For all $t$, there is an affine subspace $L = L_t \subset \RR^d$ and a Gaussian measure $\gamma_t$ on $\RR^d$, supported on $L$, with covariance
    $\Sigma_t$ such that $\mu_t$ is absolutely continuous with respect
    to $\gamma_t$, and
    \[ \frac{d\mu_t}{d\gamma_t}(x) \propto \mu(x),\ \forall x\in L. \]
  \item $\Sigma_t$ is continuous and for almost every $t$ obeys the differential equation
    \[ \frac{d}{dt} \Sigma_t = -\Sigma_t C_t^2 \Sigma_t. \]
  \item $\lim_{t \rightarrow 0^+} \Sigma_t^{-1} = 0$.
  \end{itemize}
\end{proposition}
\begin{proof}
  For $1 \le k \le d$, let $T_k$ denote the first time the measure
  $\mu_t$ is supported in a $(d - k)$-dimensional affine
  subspace, and denote by $L_t$ the affine hall of the support of $\mu_t$. We will define $\Sigma_t$ inductively for each interval
  $[T_{k-1}, T_k]$. Recall from the proof of Proposition
  \ref{prop:stochastic-localization} that $\mu_t$ is constructed by iteratively
  applying Lemma \ref{lem:stochastic-localization-prelim} to affine
  subspaces of decreasing dimension $d, d - 1, d - 2, \ldots , 1$. Let
  $\tilde{\Sigma}_{k,t}$ denote the quantity $\tilde{\Sigma}_t$, from the $k$-th application of Lemma \ref{lem:stochastic-localization-prelim}, so that $\tilde{\Sigma}_{k,t}$ is a linear operator on the subspace $L_{T_k}$.

  For the base case $0 < t \le T_1$, take $\Sigma_t =
  (\tilde{\Sigma}^{-1}_{0,t} - \Id)^{-1}$. A straightforward
  calculation shows that over this time interval,
  $\frac{d\mu_t}{d\mu}$ is proportional to the density of a Gaussian
  with covariance $\Sigma_t$. Note that since $\tilde{\Sigma}^{-1}_{0,
    0} = \Id$, we also have $\lim_{t \rightarrow 0^+} \Sigma_t^{-1} =
 0$.

  Now suppose that $\Sigma_t$ has been defined up until time $T_k$; we
  will extend it to time $T_{k+1}$. Let $L_k$ denote the affine hull
  of the support of $\mu_{T_k}$, so that $\dim(L_k) = d - k$ (if
  $\dim(L_k) < d - k$, then we simply have $T_{k+1} = T_k$). Then, for
  $0 \le t \le T_{k+1} - T_k$, we may set
  \[ \Sigma_{T_k + t} := \left( \tilde{\Sigma}^{-1}_{k,t} + \Sigma_{T_k}^{-1} - \Id \right)^{-1}, \]
  where the quantities involved are matrices over the subspace
  parallel to $L_k$ but may also be regarded as degenerate bilinear
  forms in the ambient space $\RR^d$. First, observe that continuity of the processes $\tilde{\Sigma}_{k,t}$ implies the same for $\Sigma_t$. Once again, a straightforward
  calculation shows that for $T_k \le t < T_{k+1}$,
  $\frac{d\mu_t}{d\mu}$ is proportional to the density of a Gaussian
  with covariance $\Sigma_t$, where we view $\mu_t$ and $\mu$ as
  densities on $L_k$ (for $\mu$, we take its conditional density on
  $L_k$). 

  It remains only to show that $\Sigma_t$ satisfies the required
  differential equation. From our construction, we see that $\Sigma_t$
  always takes the form $\left( \tilde{\Sigma}^{-1}_t - H
  \right)^{-1}$, where $H \preceq \Id$ and
  \[ \frac{d}{dt} \tilde{\Sigma}_t = - \tilde{\Sigma}_t C_t^2 \tilde{\Sigma}_t. \]
  Then, we have
  \begin{align*}
    \frac{d}{dt} \Sigma_t &= -\left( \tilde{\Sigma}^{-1}_t - H \right)^{-1} \left( \frac{d}{dt} \tilde{\Sigma}^{-1}_t \right) \left( \tilde{\Sigma}^{-1}_t - H \right)^{-1} \\
    &= -\Sigma_t \left(-\tilde{\Sigma}_t^{-1} \left( \frac{d}{dt} \tilde{\Sigma}_t \right) \tilde{\Sigma}_t^{-1} \right) \Sigma_t \\
    &= -\Sigma_t C_t^2 \Sigma_t,
  \end{align*}
  as desired.
\end{proof}

\begin{proposition} \label{dAt}
	$dA_t = \int\limits_{\RR^d}(x-a_t)^{\otimes 3}\mu_t(dx)C_tdB_t - A_tC_t^2A_tdt$
\end{proposition}
\begin{proof}
	We consider the Doob decomposition of $A_t = M_t + E_t$, where $M_t$ is a local martingale and $E_t$ is a process of bounded variation.
	By the previous two propositions and the definition of $A_t$, we have on one hand
	\begin{align*}
	dA_t &= d\int\limits_{\RR^d}x^{\otimes 2}\mu_t(dx) - da_t^{\otimes 2} =d\int\limits_{\RR^d}x^{\otimes 2}\mu_t(dx) -a_t\otimes da_t - da_t\otimes a_t- A_tC_t^2A_tdt.
	\end{align*}
	Clearly the first $3$ terms are local martingales, which shows, by the uniqueness of the Doob decomposition, $dE_t= -A_tC_t^2A_tdt$. On the other hand, one may also rewrite the above as
	\begin{align*}
	dA_t =& d\int\limits_{\RR^d}(x - a_t)^{\otimes 2}\mu_t(dx)  = \int\limits_{\RR^d}d\left((x - a_t)^{\otimes 2}\mu_t(dx)\right)\\
	     =& -\int\limits_{\RR^d}da_t\otimes (x-a_t)\mu_t(dx) - \int\limits_{\RR^d}(x-a_t)\otimes da_t\mu_t(dx) + \int\limits_{\RR^d}(x-a_t)^{\otimes 2} d\mu_t(dx)\\
	     &-2\int\limits_{\RR^d}(x-a_t)\otimes d[a_t,\mu_t(dx)]_t +\int\limits_{\RR^d}d[a_t,a_t]_t\mu_t(dx).
	\end{align*}
	Note that the first $2$ terms are equal to $0$, since, by definition of $a_t$,
	$$\int\limits_{\RR^d}da_t\otimes (x-a_t)\mu_t(dx) = da_t\otimes\int\limits_{\RR^d}(x-a_t)\mu_t(dx) = 0.$$
	Also, the last 2 terms are clearly of bounded variation, which shows
	$$dM_t = \int\limits_{\RR^d}(x-a_t)^{\otimes 2} d\mu_t(dx) = \int\limits_{\RR^d}(x-a_t)^{\otimes 3} C_t\mu_t(dx)dB_t.$$
\end{proof}
Define the stopping time $\tau = \inf\{t| A_t = 0\}$. Then, at time $\tau$, $\mu_\tau$ is just a delta mass located at $a_\tau$ and $\mu_s = \mu_\tau$ for every $s \geq \tau$. A crucial is observation is
\begin{proposition} \label{finite}
	Suppose that there exists constants $t_0 \geq 0$ and $c>0$ such that a.s. one of the following happens
	\begin{enumerate}
		\item for every $t_0 < t < \tau$, $\mathrm{Tr}\left(A_tC_t^2A_t\right) > c$,
		\item $\int\limits_0^{t_0}\lambda_{\min}\left(C_t^2\right)dt = \infty$, where $\lambda_{\min}\left(C_t^2\right)$ is the minimal eigenvalue of $C_t^2$,
	\end{enumerate}
	then $\tau$ is finite a.s. and in the second case $\tau \leq t_0$. Moreover, if $\tau$ is finite a.s. then  $a_\tau$ has the law of $\mu$.
\end{proposition}
\begin{proof}
	 Consider the process $R_t = A_t + \int\limits_{0}^tA_sC_s^2A_sds$. For the first case, the previous proposition shows that the real-valued process $\mathrm{Tr}\left(R_t\right)$ a positive local martingale; hence, a super-martingale. By the martingale convergence theorem $\mathrm{Tr}\left(R_t\right)$ converges to a limit almost surely. By our assumption, if $\tau = \infty$ then
	$$\int\limits_{0}^\infty\mathrm{Tr}(A_tC_t^2A_t)dt \geq  \int\limits_{t_0}^\infty\mathrm{Tr}(A_tC_t^2A_t)dt \geq\int\limits_{t_0}^\infty cdt = \infty.$$
    This would imply that $\lim\limits_{t \to \infty}\mathrm{Tr}(A_t) = -\infty$ which clearly cannot happen.
	\\
	\\
	For the second case, under the event $\{\tau > t_0\}$, by continuity of the process $A_t$ there exists $a > 0$ such that for every $t \in [0, t_0]$, there is a unit vector $v_t \in \RR^d$ for which
	$\inner{v_t}{A_tv_t} \geq a$.
	We then have,
	$$\int\limits_{0}^{t_0}\mathrm{Tr}(A_tC_t^2A_t)dt \geq \int\limits_{0}^{t_0}\inner{A_tv_t}{C_t^2A_tv_t}dt\geq a^2\int\limits_{0}^{t_0}\lambda_{\min}(C_t^2)dt  = \infty,$$
	which implies $\lim\limits_{t \to t_0}\mathrm{Tr}(A_t) = -\infty$. Again, this cannot happen and so $\PP(\tau > t_0) = 0$.
	\\
	\\
	To understand the law of $a_\tau$, let $\vphi: \RR^d \to \RR$ be any continuous bounded function. By Property \ref{item:martingale} of Proposition \ref{prop:stochastic-localization} $\int\limits_{\RR^d} \vphi(x)\mu_t(dx)$ is a martingale. We claim that it is bounded. Indeed, observe that since $\mu_t$ is a probability measure for every $t$, then $$\int\limits_{\RR^d} \vphi(x)\mu_t(dx) \leq \max\limits_x|\vphi(x)|.$$ $\tau$ is finite a.s., so by the optional stopping theorem for continuous time martingales (\cite{oksendal2003stochastic} Theorem 7.2.4)
	$$\EE\left[\int\limits_{\RR^d} \vphi(x)\mu_\tau(dx)\right] = \int\limits_{\RR^d} \vphi(x)\mu(dx).$$
	Since $\mu_\tau$ is a delta mass, we have that $\int\limits_{\RR^d} \vphi(x)\mu_\tau(dx) =\ \vphi(a_\tau)$ which finishes the proof.
\end{proof}
We finish the section with an important property of the process $A_t$.
\begin{proposition} \label{lem: decreaing rank}
	The rank of $A_t$ is monotonic decreasing in $t$, and $\ker(A_t) \subset \ker(A_s)$ for $t \leq s$.
\end{proposition}
\begin{proof}
		To see that $\mathrm{rank}(A_t)$ is indeed monotonic decreasing, let $v_0$ be such that $A_{t_0}v_0 = 0$ for some $t_0 > 0$, we will show that for any $t \geq t_0$, $A_{t}v_0 = 0$. In a similar fashion to Proposition \ref{finite}, we define the process $\inner{v_0}{A_tv_0} + \int\limits_0^t\inner{v_0}{A_sC_s^2A_sv_0}ds$, which is, using Proposition \ref{dAt}, a positive local martingale and so a super-martingale. This then implies that $\inner{v_0}{A_tv_0}$ is itself a positive super-martingale. Since $\inner{v_0}{A_{t_0}v_0} = 0$, we have that for any $t \geq t_0$, $\inner{v_0}{A_{t}v_0} = 0$ as well.\\
\end{proof}
\section{Convergence rates in transportation distance}
\subsection{The case of bounded random vectors: proof of Theorem \ref{thm:w2-bounded}}
In this subsection we fix a measure $\mu$ on $\RR^d$ and a random vector $X \sim \mu$ with the assumption that $\norm{X} \leq \beta$ almost surely for some $\beta >0$. We also assume that $\EE\left[X\right] = 0$.\\

We define the martingale process $a_t$ along with the stopping time $\tau$ as in Section \ref{construction}, where we take $C_t = A_t^\dagger$, so that $a_t = \int\limits_0^tA_sA_s^\dagger dB_s$. We denote $P_t:=A_tA_t^\dagger$, and remark that since $A_t$ is symmetric, $P_t$ is a projection matrix. As such, we have that for any $t < \tau$, $\mathrm{Tr}\left(P_t\right) \geq 1$. By Proposition \ref{finite}, $a_\tau$ has the law $\mu$.\\
\\
In light of the remark following Theorem \ref{main}, our first objective is to understand the expectation of $\tau$.
\begin{lemma} \label{boundedExpectation}
	Under the boundedness assumption $\norm{X} \leq \beta$, we have $\EE\left[\tau\right]\leq \beta^2$.
\end{lemma}
\begin{proof}
	Let $H_t = \norm{a_t}^2$. By It\^o's formula and since $P_t$ is a projection matrix,
	 $$dH_t = 2\langle a_t,P_tdB_t\rangle + \mathrm{Tr}\left(P_t\right)dt = 2\langle a_t,P_tdB_t\rangle + \text{rank}\left(P_t\right)dt.$$
	 So, $\frac{d}{dt}\EE\left[H_t\right] = \EE\left[\text{rank}\left(P_t\right)\right]$. Since $\EE\left[H_\infty\right]\leq \beta^2$,
	 $$\beta^2 \geq \EE\left[H_\infty\right]-\EE[H_0] = \int\limits_0^\infty\EE\left[\text{rank}\left(P_t\right)\right]dt\geq \int\limits_0^\infty\PP\left(\tau > t\right)dt = \EE\left[\tau\right].$$
\end{proof}
The above claim gives bounds on the expectation of $\tau$, however in order to use Theorem \ref{main}, we need bounds for its tail behaviour in the sense of \eqref{sub-exponential}. To this end, we can use a bootstrap argument and invoke the above lemma with the measure $\mu_t$ in place of $\mu$, recalling that $X_\infty|\mathcal{F}_t \sim \mu_t$ and noting that $\norm{X_\infty|\mathcal{F}_t} \leq \beta$ almost surely. Therefore, we can consider the conditioned stopping time $\tau|\mathcal{F}_t - t$ and get that
$$\EE\left[\tau|\mathcal{F}_t\right] \leq t + \beta^2.$$
The following lemma will make this precise.
\begin{lemma} \label{lem:exp tails for bdd}
	Suppose that, for the stopping time $\tau$, it holds that for every $ t > 0$, $\EE\left[\tau|\mathcal{F}_t\right] \leq t + \beta^2$ a.s., then
\begin{equation}\label{eq:exptails}
	\forall i\in \NN,\ \ \PP\left(\tau \geq i\cdot 2\beta^2\right) \leq \frac{1}{2^i}.
\end{equation}
\end{lemma}
\begin{proof}
	Denote $t_i = i\cdot 2\beta^2$. Since $\mu_t$ is Markovian, and by the law of total probability, for any $ i \in \NN$ we have the relation
	$$ \PP\left(\tau \geq t_{i+1}\right) \leq \PP\left(\tau > t_{i}\right)\mathrm{ess}\sup\limits_{\mu_{t_i}}\left(\PP\left(\tau - t_i \geq 2\beta^2|\mathcal{F}_{t_{i}}\right)\right),$$
	where the essential supremum is taken over all possible states of $\mu_{t_i}$. Using Markov's inequality, we almost surely have
	$$\PP\left(\tau - t_i \geq 2\beta^2|\mathcal{F}_{t_{i}}\right) \leq \frac{\EE\left[\tau - t_i|\mathcal{F}_{t_i}\right]}{2\beta^2} \leq \frac{1}{2},$$
	which is also true for the essential supremum.
	Clearly $\PP\left(\tau\geq 0\right) = 1$ which finishes the proof.
\end{proof}

\begin{proof}[Proof of Theorem \ref{thm:w2-bounded}]
Our objective is to apply Theorem \ref{main}, defining $X_t = a_t$ and $\Gamma_t = P_t$ so that $(X_t, \Gamma_t, \tau)$ becomes a martingale embedding according to Proposition \ref{finite}. In this case, we have that $\Gamma_t$ is a projection matrix almost surely. Thus,
$$\mathrm{Tr}\left(\EE[\Gamma_t^4]\EE\left[\Gamma_t^2\right]^\dagger\right) \leq d,$$
and
$$\mathrm{Tr}\left(\EE[\Gamma_t^2]\right) \leq d\PP\left(\tau > t\right).$$
Therefore, if $G$ and $S_n$ are defined as in Theorem \ref{main}, then
\begin{align*}
\mathcal{W}_2^2\left(S_n,G\right)&\leq \int\limits_{0}^{2\beta^2\log_2(n)}\frac{d}{n}dt + \int\limits_{2\beta^2\log_2(n)}^\infty 4d\PP(\tau > t)dt \\
&\leq \frac{2d\beta^2\log_2(n)}{n} + 4d\int\limits_{2\beta^2\log_2(n)}^\infty\PP\left(\tau > \left\lfloor\frac{t}{2\beta^2}\right\rfloor 2\beta^2\right)dt\\
& \stackrel{\eqref{eq:exptails}}{\leq} \frac{2d\beta^2\log_2(n)}{n} + 4d\int\limits_{2\beta^2\log_2(n)}^\infty\left(\frac{1}{2}\right)^{\left\lfloor\frac{t}{2\beta^2}\right\rfloor}dt\\
&\leq \frac{2d\beta^2\log_2(n)}{n} + 8d\beta^2\sum\limits_{j=\left\lfloor\log_2(n)\right\rfloor}^\infty \frac{1}{2^j} \leq \frac{2d\beta^2\log_2(n)}{n} + \frac{32d\beta^2}{n}.
\end{align*}
Taking square roots, we finally have
  \[ \mathcal{W}_2(S_n, G) \leq \frac{\beta\sqrt{d}\sqrt{32+2\log_2(n)}}{\sqrt{n}}, \]
as required.
\end{proof}
\subsection{The case of log-concave vectors: proof of Theorem \ref{thm:w2-log-concave}}
In this section we fix $\mu$ to be an isotropic log concave measure. The processes $a_t = a^\mu_t, A_t = A_t^\mu$ are defined as in Section \ref{construction} along with the stopping time $\tau$. To define the matrix process $C_t$, we first define a new stopping time $$T:= 1\wedge\inf\{t|\norm{A_t}_{op} \geq 2\}.$$
 $C_t$ is then defined in the following manner:
\[
C_t =
\begin{cases}
\mathrm{min}(A_t^\dagger, \mathrm{I}_d) & \text{if $t \leq T$} \\
A_t^\dagger & \text{otherwise}
\end{cases}
\]
where, again, $A_t^\dagger$ denotes the pseudo-inverse of $A_t$ and $\mathrm{min}(A_t^\dagger, \mathrm{I}_d)$ is the unique matrix which is diagonalizable with respect to the same basis as $A_t^\dagger$ and such that each of its eigenvalues corresponds to an an eigenvalue of  $A_t^\dagger$ truncated at $1$. Since $\mathrm{Tr}\left(A_tA_t^\dagger\right) \geq 1$ whenever $t \leq \tau$, then the conditions of Proposition \ref{finite} are clearly met for $t_0 = 1$ and $a_\tau$ has the law of $\mu$.

In order to use Theorem \ref{main}, we will also need to demonstrate that $\tau$ has subexponential tails in the sense of \eqref{sub-exponential}. For this, we first relate $\tau$ to the stopping time $T$.
\begin{lemma}
  $\tau< 1 + \frac{4}{T}$.
\end{lemma}
\begin{proof}
  Let $\Sigma_t$ be as in Proposition
  \ref{prop:sigma-evolution}. As described in the proposition, $\mu_t$ is
  proportional to $\mu$ times a Gaussian of covariance
  $\Sigma_t$, on an appropriate affine subspace. In this case, an application of the Brascamp-Lieb inequality (see
  \cite{harge2004convex} for details) shows that $A_t =
  \mathrm{Cov}(\mu_t) \preceq \Sigma_t$. In particular, this means
  that for $t > T$, when restricted to the orthogonal complement of $\ker(A_t)$, the following inequality holds,
  \[ \frac{d}{dt} \Sigma_t = -\Sigma_t C_t^2 \Sigma_t \preceq -\Id. \]
  So, $\tau \le T + \norm{\Sigma_T}_{op}$.

  It remains to estimate $\norm{\Sigma_T}_{op}$. To this end, recall
  that for $0 < t \le T$, we have $\norm{A_t}_{op} \leq 2$, which implies
  \[ \frac{d}{dt} \Sigma_t = -\Sigma_tC_t^2\Sigma_t  \preceq -\frac{1}{4} \Sigma_t^2. \]
  Now, consider the differential equation $f'(t) = -\frac{1}{4}f(t)^2$
  with $f(T) = \norm{\Sigma_T}_{op}$, which has solution $f(t) =
  \frac{4}{t - T + \frac{4}{\norm{\Sigma_T}_{op}}}$. By Gronwall's
  inequality, $f(t)$ lower bounds $\norm{\Sigma_t}_{op}$ for $0 < t
  \le T$, and so, in particular, $f(t)$ must remain finite within
  that interval. Consequently, we have
  \[ \frac{4}{\norm{\Sigma_T}_{op}} > T \implies \norm{\Sigma_T}_{op} < \frac{4}{T}. \]
  We conclude that
  \[ \tau \le T + \norm{\Sigma_T}_{op} < 1 + \frac{4}{T}, \]
  as desired.
\end{proof}

\begin{lemma} \label{expobound}
	There exist universal constants $c,C>0$ such that if $s > C\cdot \kappa_d^2\ln(d)^2$ and $d \geq 8$ then $$\PP(\tau > s)\leq e^{-cs},$$
	where $\kappa_d$ is the constant defined in \eqref{kappa_d}.
\end{lemma}
\begin{proof}
	First, by using the previous claim, we may see that for any $s \geq 5$,
	$$\PP\left(\tau > s\right) \leq \PP\left(\frac{1}{T}\geq \frac{s-1}{4}\right) \leq \PP\left(\frac{1}{T}\geq \frac{s}{5}\right) = \PP\left(5s^{-1} \geq T\right) = \PP\left(\max\limits_{0\leq t \leq 5s^{-1}}\norm{A_t}_{op} \geq 2\right).$$
	Recall from Proposition \ref{dAt},
	$$dA_t = \int\limits_{\RR^d}(x-a_t)\otimes(x-a_t)\langle C_t\left(x - a_t\right),dB_t\rangle\mu_t(dx) - A_tC_t^2A_tdt.$$
	Since we are trying to bound the operator norm of $A_t$, we might as well just consider the matrix $\tilde{A}_t =  A_t + \int\limits_0^tA_sC^2_sA_sds$. Note that, by definition of $T$, for any $t \leq T$, 
	$$\int\limits_0^tA_sC_s^2A_sds\preceq \mathrm{I}_d.$$ 
	Thus, for $t \in [0,T]$,
	\begin{equation} \label{eq: Atilde bound}
	3\mathrm{I}_d \succeq A_t + \mathrm{I}_d\succeq \tilde{A}_t \succeq A_t.
	\end{equation}
	Also, $\tilde{A}_t$ can be written as, 
	\begin{equation} \label{Atilde}
	d\tilde{A}_t = \int\limits_{\RR^d}(x-a_t)\otimes(x-a_t)\langle C_t(x-a_t),dB_t\rangle\mu_t(dx),\ \tilde{A}_0 = \mathrm{I}_d.
	\end{equation}
	The above shows
	 $$\PP\left(\max\limits_{0\leq t \leq 5s^{-1}}\norm{A_t}_{op} \geq 2\right) \leq \PP\left(\max\limits_{0\leq t \leq 5s^{-1}}||\tilde{A}_t||_{op} \geq 2\right).$$
	We note than whenever $||\tilde{A}_t||_{op} \geq 2$ then also $\mathrm{Tr}\left(\tilde{A}_t^{4\ln(d)}\right)^{\frac{1}{4\ln(d)}} \geq 2$, so that
	\begin{align} \label{doobdec}
	\PP&\left(\max\limits_{0\leq t \leq 5s^{-1}}||\tilde{A}_t||_{op} \geq 2\right) \leq \PP\left(\max\limits_{0\leq t \leq 5s^{-1}}\mathrm{Tr}\left(\tilde{A}_t^{4\ln(d)}\right)^{\frac{1}{4\ln(d)}} \geq 2\right)\nonumber\\
	&\leq \PP\left(\max\limits_{0\leq t \leq 5s^{-1}}\ln\left(\mathrm{Tr}\left(\tilde{A}_t^{4\ln(d)}\right)\right) \geq 2\ln(d)\right) = \PP\left(\max\limits_{0\leq t \leq 5s^{-1}}\left(M_t + E_t\right) \geq 2\ln(d)\right),
	\end{align}
	where $M_t$ and $E_t$ form the Doob-decomposition of $\ln\left(\mathrm{Tr}\left(\tilde{A}_t^{4\ln(d)}\right)\right)$. That is, $M_t$ is a local martingale and $E_t$ is a process of bounded variation. To calculate the differential of the Doob-decomposition, fix $t$,
	let $v_1,...,v_n$ be the unit eigenvectors of $\tilde{A}_t$ and let $\alpha_{i,j} = \langle v_i, \tilde{A}_tv_j\rangle$ with
	\begin{align*}
	d\alpha_{i,j} = \int\limits_{\RR^d}\langle x,v_i\rangle\langle x,v_j\rangle\langle C_tx,dB_t\rangle\mu_t(dx + a_t),
	\end{align*}
	which follows from \eqref{Atilde}. Also define
	$$\xi_{i,j} = \frac{1}{\sqrt{\alpha_{i,i}\alpha_{j,j}}}\int\limits_{\RR^d}\langle x,v_i\rangle\langle x,v_j\rangle C_tx\mu_t(dx+ a_t).$$
	So that
	$$d\alpha_{i,j} =  \sqrt{\alpha_{i,i}\alpha_{j,j}}\inner{\xi_{i,j}}{dB_t},\ \ \ \frac{d}{dt}[\alpha_{i,j}]_t = \alpha_{i,i}\alpha_{j,j}\norm{\xi_{i,j}}^2.$$
	Now, since $v_i$ is an eigenvector corresponding to the eigenvalue $\alpha_{i,i}$, we have
	$$\xi_{i,j} = \int\limits_{\RR^d}\langle \tilde{A}_t^{-1/2}x,v_i\rangle\langle \tilde{A}_t^{-1/2}x,v_j\rangle C_tx\mu_t(dx+a_t).$$
	If we define the measure $\tilde{\mu}_t(dx) = \det(\tilde{A}_t)^{1/2}\mu_t(\tilde{A}_t^{1/2}dx + a_t)$, then $\tilde{\mu}_t$ has the law of a centered log-concave random vector with covariance $\tilde{A}_t^{-1/2}A_t\tilde{A}_t^{-1/2}\preceq \mathrm{I}_d$. By making the substitution $y = \tilde{A}_t^{-1/2}x$, the above expression becomes
	$$\xi_{i,j}=\int\limits_{\RR^d} \inner{y}{v_i}\inner{y}{v_j}C_t\tilde{A}_t^{1/2}y\tilde{\mu}_t(dy).$$
	By \eqref{eq: Atilde bound} and the definition of $T$, $C_t$, for any $t \leq T$, $\tilde{A}_t^{1/2} \preceq 2\mathrm{I}_d$ and $C_t \preceq \mathrm{I}_d$. So, $\norm{C_t\tilde{A}_t^{1/2}}_{op}\leq 2$. Under similar conditions, it was shown in \cite{eldan2013thin}, Lemma 3.2, that there exists a universal constant $C >0$ for which
	\begin{itemize}
		\item for any $1 \leq i \leq d$, $\norm{\xi_{i,i}}^2\leq C$.
		\item for any $1 \leq i \leq d$, $\sum\limits_{j=1}^d\norm{\xi_{i,j}}^2\leq C\kappa_d^2$.
	\end{itemize}
	Furthermore, in the proof of Proposition 3.1 in the same paper it was shown
	\begin{align*}
	d\mathrm{Tr}\left(\tilde{A}_t^{4\ln(d)}\right) \leq 4\ln(d)\sum\limits_{i=1}^d\alpha_{i,i}^{4\ln(d)}\inner{\xi_{i,i}}{dB_t} + 16C\kappa_d^2\ln(d)^2\mathrm{Tr}\left(\tilde{A}_t^{4\ln(d)}\right)dt.
	\end{align*}
	So, using It\^o's formula with the function $\ln(x)$ we can calculate the differential of the Doob decomposition \eqref{doobdec}. Specifically, we use the fact that the second derivative of $\ln(x)$ is negative and get
	$$dE_t \leq 16C\kappa_d^2\ln(d)^2\frac{\mathrm{Tr}\left(\tilde{A}_t^{4\ln(d)}\right)}{\mathrm{Tr}\left(\tilde{A}_t^{4\ln(d)}\right)} = 16C\kappa_d^2\ln(d)^2, \ E_0 = \ln(d),$$
	and
\begin{equation} \label{eq:dMt}
	\frac{d}{dt}[M]_t \leq 16C^2\ln(d)^2\left(\frac{\mathrm{Tr}\left(\tilde{A}_t^{4\ln(d)}\right)}{\mathrm{Tr}\left(\tilde{A}_t^{4\ln(d)}\right)}\right)^2 = 16C^2\ln(d)^2.
\end{equation}
	Hence, $E_t \leq t\cdot 16C\kappa_n^2\ln(d)^2 + \ln(d)$, which together with \eqref{doobdec} gives
	$$\PP\left(\tau > s\right) \leq \PP\left(\max\limits_{0\leq t \leq 5s^{-1}}M_t \geq 2\ln(d) - \ln(d) - 80s^{-1}C\kappa_d^2\ln(d)^2\right)\ \forall s\geq5.$$
	Under the assumption $s > 80C\kappa_d^2\ln(d)^2$, and since $d \geq 8$, the above can simplify to
	\begin{equation}\label{prob-bound}
	\PP\left(\tau > s\right) \leq \PP\left(\max\limits_{0\leq t \leq 5s^{-1}}M_t \geq \frac{1}{2}\ln(d)\right).
	\end{equation}
	To bound this last expression, we will apply the Dubins-Schwartz theorem to write
	$$M_t = W_{[M]_t},$$
	where $W_t$ is some Brownian motion. 
	Combining this with \eqref{prob-bound} gives
	\begin{align*}
	\PP\left(\tau > s\right)& \leq \PP\left(\max\limits_{0\leq t \leq 5s^{-1}}W_{[M]_t} \geq \frac{\ln(d)}{2}\right).
	\end{align*}
	An application of Doob's maximal inequality (\cite{revuz2013continuous} Proposition I.1.8) shows that for any $t',K>0$
	$$\PP\left(\max\limits_{0\leq t \leq t'}W_t \geq K\right)\leq \exp\left(-\frac{K^2}{2t'}\right).$$
	We now integrate \eqref{eq:dMt} and use the above inequality to obtain
	$$\PP\left(\max\limits_{0\leq t \leq 5s^{-1}}W_{[M]_t} \geq \frac{\ln(d)}{2}\right) \leq e^{-cs},$$
	where $c > 0$ is some universal constant.	
\end{proof}
\begin{proof}[Proof of Theorem \ref{thm:w2-log-concave}]

By definition of $T$ and $C_t$, we have that for any $t \leq T$, $A_tC_t \preceq 2\mathrm{I}_d$ and for any $t > T$,  $A_tC_t = A_tA_t^\dagger \preceq \mathrm{I}_d$. We now invoke Theorem \ref{main}, with $\Gamma_t = A_tC_t$, for which
$$\mathrm{Tr}\left(\EE[\Gamma_t^4]\EE\left[\Gamma_t^2\right]^
\dagger\right)\leq 4d,$$
and, by Lemma \ref{expobound}
$$\mathrm{Tr}\left(\EE[\Gamma_t^2]\right) \leq 4d\PP\left(\tau > t\right)\leq 4de^{-ct}\ \ \ \ \forall t > C\cdot \kappa_d^2\ln(d)^2.$$
If $G$ is the standard $d$-dimensional Gaussian, then the theorem yields
\begin{align*}
\mathcal{W}_2^2(S_n, G) &\leq \int\limits_0^{C\cdot \kappa_d^2\ln(d)^2\ln(n)} 4\frac{d}{n}dt + \int\limits_{C\cdot \kappa_d^2\ln(d)^2\ln(n)}^\infty 16d\PP\left(\tau > t\right)\\
&\leq 4\frac{dC\cdot \kappa_d^2\ln(d)^2\ln(n)}{n} + 16d\int\limits_{C\cdot \kappa_d^2\ln(d)^2\ln(n)}^\infty e^{-ct}dt \\
&\leq C'\frac{d\cdot \kappa_d^2\ln(d)^2\ln(n)}{n}.
\end{align*}
Thus
  \[ \mathcal{W}_2(S_n, G) \leq \frac{C\kappa_d\ln(d)\sqrt{d\ln(n)}}{\sqrt{n}}, \]

\end{proof}
\section{Convergence rates in entropy}
Throughout this section, we fix a centered measure $\mu$ on $\RR^d$ with an invertible covariance matrix $\Sigma$ and $G \sim \mathcal{N}\left(0, \Sigma\right)$. Let $\{X^{(i)}\}$ be independent copies of $X \sim \mu$ and $S_n := \frac{1}{\sqrt{n}}\sum\limits_{i=1}^{n}X^{(i)}$.

Our goal is to study the quantity $\text{Ent}\left(S_n||G\right)$. In light of Theorem \ref{thm: quant entropic clt}, we aim to construct a martingale embedding $(X_t, \Gamma_t, 1)$ such that $X_1 \sim \mu$ and which satisfies appropriate bounds on the matrix $\Gamma_t$. Our construction uses the process $a_t$ from Proposition \ref{prop:stochastic-localization} with the choice $C_t := \frac{1}{1-t} \mathrm{I}_d$. Property \ref{item:da_t} in Proposition \ref{prop:stochastic-localization} gives
\begin{equation*} \label{derivativeat}
a_t = \int\limits_{0}^t\frac{A_s}{1-s}dB_s.
\end{equation*}
Thus, we denote
\begin{equation*}
\Gamma_t := \frac{A_t}{1-t}.
\end{equation*}
Since $\int\limits_0^1 \lambda_{min} (C_t^2) = \infty$, Proposition \ref{finite} shows that the triplet $(a_t, \Gamma_t, 1)$ is a martingale embedding of $\mu$. As above, the sequence $\Gamma_t^{(i)}$ will denote independent copies of $\Gamma_t$ and we define $\tilde{\Gamma}_t := \sqrt{\sum_{i=1}^n \left(\Gamma_t^{(i)}\right)^2}$.\\
\\
\subsection{Properties of the embedding}
The martingale embedding has several useful properties which we record in this section.
First, we give an alternative description of the process which will be of use for us. Define the random process
$$v  := \arg\min\limits_u \frac{1}{2}\int\limits_{0}^1\EE\left[\norm{u_t}^2\right],$$
where $u$ varies over all $\FF_t$-adapted drifts such that $B_1 + \int\limits_{0}^1u_tdt \sim \mu$. Denote $$Y_t := B_t + \int\limits_{0}^t v_sds.$$
In \cite{eldan2018regularization} (Section 2.2) it was shown that the density of the measure $Y_1|\mathcal{F}_t$ has the same dynamics as the density of $\mu_t$. Thus, almost surely $Y_1|\mathcal{F}_t \sim \mu_t$ and since $a_t$ is the expectation of $\mu_t$, we have the identity
\begin{equation} \label{conditional at}
a_t = \EE\left[Y_1| \mathcal{F}_t \right],
\end{equation}
and in particular we have $a_1 = Y_1$. Moreover, the same reasoning implies that $A_t = \mathrm{Cov}(Y_1|\mathcal{F}_t)$ and
\begin{equation} \label{gamma_var}
\Gamma_t = \frac{\mathrm{Cov}(Y_1|\mathcal{F}_t)}{1-t}.
\end{equation}
The process $Y_t$ goes back at least to the works of F\"ollmer \cite{follmer1985entropy, follmer1986time}. In a later work, by Lehec \cite{lehec2013representation}, it is shown that $v_t$ is a martingale and that
\begin{equation} \label {variational entropy}
\mathrm{Ent}(Y_1||\gamma) = \frac{1}{2}\int\limits_0^1\EE\left[\norm{v_t}^2\right]dt,
\end{equation}
where $\gamma$ denotes the standard Gaussian.
\begin{lemma} \label{derivative of gamma}
It holds that $\frac{d}{dt}\EE\left[\mathrm{Cov}(Y_1|\mathcal{F}_t)\right] = -\EE\left[\Gamma_t^2\right].$
\end{lemma}
\begin{proof}
From \eqref{conditional at}, we have
$$\mathrm{Cov}(Y_1|\mathcal{F}_t) = \EE\left[Y_1^{\otimes 2} |\mathcal{F}_t\right] - \EE\left[Y_1 |\mathcal{F}_t\right]^{\otimes 2} =\EE\left[Y_1^{\otimes 2} |\mathcal{F}_t\right] - a_t^{\otimes 2}.$$
$a_t$ is a martingale, hence
\begin{equation}
\frac{d}{dt}\EE\left[\mathrm{Cov}(Y_1|\mathcal{F}_t)\right] = -\frac{d}{dt}\EE\left[[a]_t\right] = -\EE\left[\Gamma_t^2\right].
\end{equation}
\end{proof}
Our next goal is to recover $v_t$ from the martingale $a_t$.
\begin{lemma} \label{v as gamma}
The drift $v_t$ satisfies that identity $v_t = \int\limits_0^t\frac{\Gamma_s - \mathrm{I}_d}{1-s}dB_s.$
Furthermore,
\begin{equation}\label{eq:vtexpint}
\EE \left [\norm{v_t}^2 \right ] = \int\limits_0^t\frac{\mathrm{Tr}\left(\EE\left[\left(\Gamma_s - \mathrm{I}_d\right)^2\right]\right)}{(1-s)^2}ds.
\end{equation}
\end{lemma}
\begin{proof}
 We begin by writing
$$da_t = dB_t + \left(\Gamma_t - \mathrm{I}_d\right)dB_t.$$
Using Fubini's theorem then yields
$$\int\limits_0^1\left(\Gamma_s - \mathrm{I}_d\right)dB_s = \int\limits_0^1\int\limits_s^1\frac{\Gamma_s - \mathrm{I}_d}{1-s}dtdB_s = \int\limits_{0}^1\int\limits_0^t\frac{\Gamma_s - \mathrm{I}_d}{1-s}dB_sdt.$$
Therefore, defining $\tilde v_t = \int\limits_0^t\frac{\Gamma_s - \mathrm{I}_d}{1-s}dB_s$ we have that $\tilde v_t$ is a martingale. 
 and that $B_1 + \int\limits_0^1 \tilde v_t dt = a_1$. It follows that $v_t - \tilde v_t$ is a martingale and that $\int\limits_0^1 (v_t - \tilde v_t) dt = 0$. We will now show that if a martingale $Q_t$ satisfies $Q_0 = 0$ and $\int\limits_{0}^1Q_tdt = 0$ a.s., then $Q_t = 0$ for every $t\in [0,1]$. From this, it will follow that $v_t = \tilde{v}_t$. Indeed, write $Q_t = \int\limits_0^t Q'_sdB_s$, for some adapted process $Q'_t$. Using Fubini's theorem, a calculation, similar to the one above, gives the identity,
 $$0 = \int\limits_0^1Q_tdt = \int\limits_0^1 (1-t)Q'_tdB_t.$$
 Considering the martingale $\int\limits_0^\cdot (1-t)Q'_tdB_t$, we now have, for any $s \in [0,1)$
 $$0 = \EE\left[\int\limits_0^1 (1-t)Q'_tdB_t|\mathcal{F}_s\right] = \int\limits_0^s (1-t)Q'_tdB_t.$$
 Thus, $Q' = 0$ almost surely, which implies, for every $t \in [0,1]$,  $Q_t = Q_0 = 0.$ Therefore $v_t = \tilde v_t$, or in other words
\begin{equation*}
 v_t = \int\limits_0^t\frac{\Gamma_s - \mathrm{I}_d}{1-s}dB_s.
\end{equation*}
Finally, equation \eqref{eq:vtexpint} follows from a direct application of It\^{o}'s isometry.
\end{proof}
A combination of equations \eqref{variational entropy} and \eqref{eq:vtexpint} gives the useful identity,
\begin{align} \label{eq: variational entropy}
\mathrm{Ent}\left(Y_1||\gamma\right) = \frac{1}{2}\int\limits_0^1\int\limits_0^t\frac{\mathrm{Tr}\left(\EE\left[\left(\Gamma_s - \mathrm{I}_d\right)^2\right]\right)}{(1-s)^2}dsdt =\frac{1}{2}\int\limits_0^1\frac{\mathrm{Tr}\left(\EE\left[\left(\Gamma_t - \mathrm{I}_d\right)^2\right]\right)}{1-t}dt.
\end{align}
The above lemma also affords a representation of $\EE\left[\mathrm{Tr}\left(\Gamma_t\right)\right]$ in terms of $\EE\left[\norm{v_t}^2\right]$.
\begin{lemma}\label{lem:gamma representation}
	 It holds that
	$$\EE\left[\mathrm{Tr}(\Gamma_t)\right] =  d - (1-t)\left(d - \mathrm{Tr}\left(\Sigma\right) + \EE\left[\norm{v_t}^2\right]\right).$$
\end{lemma}
\begin{proof}
	The identity can be obtained through integration by parts. By Lemma \ref{v as gamma},
	\begin{align*}
	\EE[\norm{v_t}^2] & \stackrel{\eqref{eq:vtexpint}}{=} \int\limits_0^t\frac{\mathrm{Tr}\left(\EE\left[\left(\Gamma_s - \mathrm{I}_d\right)^2\right]\right)}{(1-s)^2}ds\\
	&=\int\limits_0^t\frac{\mathrm{Tr}\left(\EE\left[\Gamma_s^2\right]\right)}{(1-s)^2}ds - 2\int\limits_0^t\frac{\mathrm{Tr}\left(\EE\left[\Gamma_s\right]\right)}{(1-s)^2}ds + \int\limits_0^t\frac{\mathrm{Tr}\left(\mathrm{I}_d\right)}{(1-s)^2}ds.
	\end{align*}
	Since, by Lemma \ref{derivative of gamma}, $\frac{d}{dt}\EE\left[\mathrm{Cov}\left(Y_1|\mathcal{F}_t\right)\right] = -\EE\left[\Gamma_t^2\right]$ integration by parts shows
	\begin{align*}
	\int\limits_0^t\frac{\mathrm{Tr}\left(\EE\left[\Gamma_s^2\right]\right)}{(1-s)^2}ds &= -\frac{\mathrm{Tr}\left(\EE\left[\mathrm{Cov}\left(Y_1|\mathcal{F}_s\right)\right]\right)}{(1-s)^2}\Bigg\vert^t_0 + 2\int\limits_0^t\frac{\mathrm{Tr}\left(\EE\left[\mathrm{Cov}\left(Y_1|\mathcal{F}_s\right)\right]\right)}{(1-s)^3}ds\\
	&=\mathrm{Tr}\left(\Sigma\right) -  \frac{\mathrm{Tr}\left(\EE\left[\Gamma_t\right]\right)}{1-t} + 2\int\limits_0^t\frac{\mathrm{Tr}\left(\EE\left[\Gamma_s\right]\right)}{(1-s)^2}ds,
	\end{align*}
	where we have used \eqref{gamma_var} and the fact $\mathrm{Cov}\left(Y_1|\mathcal{F}_0\right) = \mathrm{Cov}\left(Y_1\right) = \Sigma$. Plugging this into the previous equation shows
	$$\EE[\norm{v_t}^2] = \mathrm{Tr}\left(\Sigma\right) -  \frac{\mathrm{Tr}\left(\EE\left[\Gamma_t\right]\right)}{1-t} +\frac{d}{1-t}-d.$$
	or equivalently
	$$\EE\left[\mathrm{Tr}(\Gamma_t)\right] =  d - (1-t)\left(d - \mathrm{Tr}\left(\Sigma\right) + \EE\left[\norm{v_t}^2\right]\right).$$
\end{proof}
Next, as in Theorem \ref{thm: quant entropic clt}, we define $\sigma_t$ to be the minimal eigenvalue of $\EE\left[\Gamma_t\right]$, so that
$$\EE\left[\Gamma_t\right] \succeq \sigma_t\mathrm{I}_d.$$
Note that by Jensen's inequality we also have
\begin{equation} \label{sigma jensen}
\EE\left[\Gamma_t^2\right] \succeq \sigma_t^2\mathrm{I}_d.
\end{equation}
\begin{lemma} \label {lem:invertible gamma}
	Assume that $\Ent(Y_1||\gamma) < \infty$. Then $\Gamma_t$ is almost surely invertible for all $t \in [0,1)$ and, moreover, there exists a constant $m = m_\mu > 0$ for which
	$$\sigma_t \geq m, ~~ \forall t \in [0,1).$$
\end{lemma}
\begin{proof}
We will show that for every $0\leq t <1$, $\sigma_t > 0$ and that there exists $c>0$ such that $\sigma_t > \frac{1}{8}$ whenever $t>1-c$. The claim will then follow by continuity of $\sigma_t$. The key to showing this is identity \eqref{eq: variational entropy}, due to which,
$$
\mathrm{Ent}\left(Y_1||\gamma\right) = \frac{1}{2}\int\limits_0^1\frac{\mathrm{Tr}\left(\EE\left[\left(\Gamma_t - \mathrm{I}_d\right)^2\right]\right)}{1-t}dt.
$$
Recall that, by Equation \eqref{gamma_var}, $\Gamma_t = \frac{\mathrm{Cov}\left(Y_1|\mathcal{F}_t\right)}{1-t}$ and observe that, by Proposition \ref{lem: decreaing rank}, if  $\mathrm{Cov}\left(Y_1|\mathcal{F}_s\right)$ is not invertible for some $0 \leq s < 1$ then $\mathrm{Cov}\left(Y_1|\mathcal{F}_t\right)$ is also not invertible for any $t > s$. Under this event, we would have that $\int\limits_s^1\frac{\mathrm{Tr}\left(\left(\Gamma_t - \mathrm{I}_d\right)^2\right)}{1-t}dt = \infty$ which, using the above display, implies that the probability of this event must be zero. Therefore, $\Gamma_t$ is almost surely invertible and $\sigma_t > 0$ for all $t \in [0,1)$.\\
\\
Suppose now that for some $t' \in [0,1]$, $\sigma_{t'} \leq \frac{1}{8}$. By Jensen's inequality, we have
$$\mathrm{Tr}\left(\EE\left[\left(\Gamma_t - \mathrm{I}_d\right)^2\right]\right) \geq \mathrm{Tr}\left(\EE\left[\Gamma_t - \mathrm{I}_d\right]^2\right) \geq (1-\sigma_t)^2 \geq 1 - 2\sigma_t.$$
Since, by Lemma \ref{derivative of gamma}, $\EE\left[\mathrm{Cov}\left(Y_1|\mathcal{F}_t\right)\right]$ is non increasing, for any $t' \leq t \leq t' + \frac{1-t'}{2}$,
$$\sigma_t \leq \frac{\sigma_{t'}(1-t')}{1-t} \leq \frac{1-t'}{8(1- t' - \frac{1-t'}{2})} = \frac{1}{4}.$$

Now, assume by contradiction that there exists a sequence $ t_i \in (0,1)$ such that $\sigma_{t_i} \leq \frac{1}{8}$ and $\lim\limits_{i \to \infty}t_i = 1$. By passing to a subsequence we may assume that $t_{i+1} - t_{i} \geq \frac{1 - t_i}{2}$ for all $i$. The assumption $\Ent(Y_1||\gamma) < \infty$ combined with Equation \eqref{eq: variational entropy} and with the last two displays finally gives
$$\infty  > \int\limits_0^1\frac{\mathrm{Tr}\left(\EE\left[\left(\Gamma_t - \mathrm{I}_d\right)^2\right]\right)}{1-t}dt \geq \int\limits_0^1\frac{1-2\sigma_t}{1-t}dt  \geq \sum\limits_{i=1}^\infty \int\limits_{t_i}^{t_i + \frac{1-t_i}{2}}\frac{1}{2}\frac{1}{1-t}dt \geq \log 2 \sum\limits_{i=1}^\infty \frac{1}{2},$$
which leads to a contradiction and completes the proof.
\end{proof}
\subsection{Proof of Theorem \ref{thm:entropic clt}}
Thanks to the assumption $\mathrm{Ent}\left(Y_1||G\right) < \infty$, an application of Lemma \ref{lem:invertible gamma} gives that $\Gamma_t$ is invertible almost surely, so we may invoke the second bound in Theorem \ref{thm: quant entropic clt} to obtain
\begin{align*}
\mathrm{Ent}( S_n||G)\leq
\int\limits_0^1 \frac{\mathrm{Tr}\left(\EE\left[\Gamma_t^2\right] - \EE\left[\tilde{\Gamma}_t\right]^2\right)}{(1-t)^2}\left(\int\limits_t^1\sigma_s^{-2}ds\right)dt.
\end{align*}
The same lemma also shows that for some $m>0$ one has
$$\int\limits_t^1\sigma_s^{-2}ds \leq \frac{1-t}{m^2}. $$
Therefore, we attain that
\begin{align} \label{eq: integral}
\mathrm{Ent}( S_n||G)\leq
\frac{1}{m^2}\int\limits_0^1 \frac{\mathrm{Tr}\left(\EE\left[\Gamma_t^2\right] - \EE\left[\tilde{\Gamma}_t\right]^2\right)}{1-t}dt.
\end{align}
Next, observe that, by It\^o's isometry,
$$\mathrm{Cov}(X) = \int\limits_0^1\EE\left[\Gamma_t^2\right]dt.$$
Hence, as long as $\mathrm{Cov}(X)$ is finite, $\EE\left[\Gamma_t^2\right]$ is also finite for all $t \in A$ where $[0,1] \setminus A$ is a set of measure $0$. We will use this fact to show that
\begin{equation}\label{eq:pointwise}
\lim\limits_{n \to \infty}\mathrm{Tr}\left(\EE\left[\Gamma_t^2\right] -\EE\left[\tilde{\Gamma_t}\right]^2\right) = 0, ~~ \forall t \in A.
\end{equation}
Indeed, by the law of large numbers, $\tilde{\Gamma}_t$ almost surely converges  to $\sqrt{\EE\left[\Gamma_t^2\right]}$. Since $\left(\Gamma_t^{(i)}\right)^2$ are integrable, we get that the sequence $\frac{1}{n}\sum\limits_{i=1}^n\left(\Gamma_t^{(i)}\right)^2$ is uniformly integrable. We now use the inequality
$$\tilde{\Gamma}_t \preceq \sqrt{\frac{1}{n}\sum\limits_{i=1}^n\left(\Gamma_t^{(i)}\right)^2+\mathrm{I}_d} \preceq \frac{1}{n}\sum\limits_{i=1}^n\left(\Gamma_t^{(i)}\right)^2+\mathrm{I}_d,$$
to deduce that $\tilde{\Gamma_t}$ is uniformly integrable as well.
An application of Vitali's convergence theorem (see \cite{folland99real}, for example) implies \eqref{eq:pointwise}.\\

We now know that the integrand in the right hand side of \eqref{eq: integral} convergence to zero for almost every $t$. It remains to show that the expression converges as an integral, for which we intend to apply the dominated convergence theorem. It thus remains to show that the expression
$$
\frac{\mathrm{Tr}\left(\EE\left[\Gamma_t^2\right] - \EE\left[\tilde{\Gamma}_t\right]^2\right)}{1-t}
$$
is bounded by an integrable function, uniformly in $n$, which would imply that
$$
\lim\limits_{n \to \infty}\mathrm{Ent}(S_n||G) = 0,
$$
and the proof would be complete. To that end, recall that the square root function is concave on positive definite matrices (see e.g., \cite{ando1979concavity}), thus $$\tilde{\Gamma_t} \succeq \frac{1}{n}\sum\limits_{i=1}^n\Gamma_t^{(i)}.$$
It follows that
$$\mathrm{Tr}\left(\EE\left[\Gamma_t^2\right] -\EE\left[\tilde{\Gamma_t}\right]^2\right) \leq \mathrm{Tr}\left(\EE\left[\Gamma_t^2\right] -\EE\left[\Gamma_t\right]^2\right) \leq \mathrm{Tr}\left(\EE\left[\left(\Gamma_t -\mathrm{I}_d\right)^2\right]\right).$$
So we have
\begin{align*}
\frac{1}{m^2}\int\limits_0^1 \frac{\mathrm{Tr}\left(\EE\left[\Gamma_t^2\right] - \EE\left[\tilde{\Gamma}_t\right]^2\right)}{1-t}dt ~& \leq \frac{1}{m^2}\int\limits_0^1\frac{\mathrm{Tr}\left(\EE\left[\left(\Gamma_t - \mathrm{I}_d\right)^2\right]\right)}{1-t}dt \\
& \stackrel{\eqref{eq: variational entropy}}{=} \frac{2}{m^2}\mathrm{Ent}\left(Y_1||\gamma\right)< \infty.
\end{align*}
This completes the proof.

\subsection{Quantitative bounds for log concave random vectors}
In this section, we make the additional assumption that the measure $\mu$ is log concave. Under this assumption, we show how one can obtain explicit convergence rates in the central limit theorem. Our aim is to use the bound in Theorem \ref{thm: quant entropic clt} for which we are required to obtain bounds on the process $\Gamma_t$. We begin by recording several useful facts concerning this process.
\begin{lemma} \label{gamma_properties}
	The process $\Gamma_t$ has the following properties:
	\begin{enumerate}
		\item  If $\mu$ is log concave, then for every $t \in [0,1]$, $\Gamma_t \preceq \frac{1}{t}\mathrm{I}_d$, almost surely.
		\item  If $\mu$ is also $1$-uniformly log concave, then for every $t \in [0,1]$, $\Gamma_t \preceq \mathrm{I}_d$ almost surely.
	\end{enumerate}
\end{lemma}
\begin{proof}
  Denote by $\rho_t$ the density of $Y_1|\mathcal{F}_t$ with respect
  to the Lebesgue measure with $\rho := \rho_0$ being the density of
  $\mu$. By Proposition \ref{prop:sigma-evolution} with $C_t =
  \frac{\Id}{1 - t}$, we can calculate the ratio between $\rho_t$ and
  $\rho$. In particular, we have
  \[ \frac{d}{dt} \Sigma_t^{-1} = -\Sigma_t^{-1} \left( \frac{d}{dt} \Sigma_t \right) \Sigma_t^{-1} = \frac{1}{(1 - t)^2} \Id. \]
  Solving this differential equation with the initial condition
  $\Sigma_0^{-1} = 0$, we find that $\Sigma_t^{-1} = \frac{t}{1 - t}
  \Id$.

  Since the ratio between $\rho_t$ and $\rho$ is proportional to the
  density of a Gaussian with covariance $\Sigma_t$, we thus have
  \[ -\nabla^2\log(\rho_t) = -\nabla^2\log(\rho) + \frac{t}{1-t}\mathrm{I}_d. \]
  Now, if $\mu$ is log concave then $Y_1|\mathcal{F}_t$ is almost
  surely $\frac{t}{1-t}$-uniformly log-concave. By the Brascamp-Lieb
  inequality (as in \cite{harge2004convex}) we get
  $\mathrm{Cov}\left(Y_1|\mathcal{F}_t\right) \preceq
  \frac{1-t}{t}\mathrm{I}_d$ and, using \eqref{gamma_var},
  \[ \Gamma_t \preceq \frac{1}{t}\mathrm{I}_d. \]
  If $\mu$ is also $1$-uniformly log-concave then $-\nabla^2\log(\rho)
  \succeq \mathrm{I}_d$ and almost surely
  \[ -\nabla^2\log(\rho_t) \succeq \frac{1}{1-t}\mathrm{I}_d. \]
  By the same argument this implies
  \[ \Gamma_t \preceq \mathrm{I}_d. \]
\end{proof}

The relative entropy to the Gaussian of a log concave measure with non-degenerate covariance structure is finite (it is even universally bounded, see \cite{marsiglietti2018lower}). Thus, by Lemma \ref{lem:invertible gamma}, it follows that $\Gamma_t$ is invertible almost surely. This allows us to invoke the first bound of Theorem \ref{thm: quant entropic clt},
\begin{equation} \label{entropy integral bonund}
\mathrm{Ent}(S_n|| G)\leq \frac{1}{n}\int\limits_0^1\frac{\EE\left[\mathrm{Tr}\left(\left(\Gamma_t^2 - \EE\left[\Gamma_t^2\right]\right)^2\right)\right]}{(1-t)^2\sigma_t^2}\left(\int\limits_t^1\sigma_s^{-2}ds\right)dt.
\end{equation}
Attaining an upper bound on the right hand side amounts to a concentration estimate for the process $\Gamma_t^2$ and a lower bound on $\sigma_t$. These two tasks are the objective of the following two lemmas.
\begin{lemma} \label{trace bounds}
If $\mu$ is log concave and isotropic then for any $t\in[0,1)$,
$$\mathrm{Tr}\left(\EE\left[\left(\Gamma_t^2 - \EE\left[\Gamma_t^2\right]\right)^2\right]\right) \leq \frac{1-t}{t^2}\left(\frac{d(1+t)}{t^2} + 2\EE\left[\norm{v_t}^2\right]\right),$$
and
$$\mathrm{Tr}\left(\EE\left[\left(\Gamma_t^2 - \EE\left[\Gamma_t^2\right]\right)^2\right]\right) \leq C\frac{d^4}{(1-t)^4}$$
for a universal constant $C > 0$.
\end{lemma}
\begin{proof}
	The isotropicity of $\mu$, used in conjunction with the formula given in Lemma \ref{lem:gamma representation}, yields
	$$\mathrm{Tr}\left(\EE\left[\Gamma_t^2\right]\right) \geq \frac{1}{d} \mathrm{Tr}\left(\EE\left[\Gamma_t\right]\right)^2 \geq d - 2(1-t)\EE\left[\norm{v_t}^2\right],$$
	where the first inequality follows by convexity.
	Since $\mu$ is log concave, Lemma \ref{gamma_properties} ensures that, almost surely, $\Gamma_t \preceq \frac{1}{t}\mathrm{I}_d$. Therefore,
	\begin{align*}
	\mathrm{Tr}\left(\EE\left[\left(\Gamma_t^2 - \EE\left[\Gamma_t^2\right]\right)^2\right]\right) &\leq \mathrm{Tr}\left(\EE\left[\left(\Gamma_t^2 - \frac{1}{t^2}\mathrm{I}_d\right)^2\right]\right) \\
	& = \frac{1}{t^4} \mathrm{Tr}\left(\EE\left[\left(\mathrm{I}_d - t^2 \Gamma_t^2 \right)^2 \right ] \right ) \\
	& \leq \frac{1}{t^4}\mathrm{Tr}\left(\EE\left[\mathrm{I}_d - t^2 \Gamma_t^2\right]\right)\\
	&\leq \frac{1-t}{t^2}\left(\frac{d(1+t)}{t^2} + 2\EE\left[\norm{v_t}^2\right]\right).
	\end{align*}
	Which proves the first bound. Towards the second bound, we use \eqref{gamma_var} to write
	$$\Gamma_t^2\preceq \frac{1}{(1-t)^2}\EE\left[Y_1^{\otimes 2}|\mathcal{F}_t\right]^2.$$
	So,
	$$\EE\left[\norm{\Gamma_t^2}_{HS}^2\right] \ \leq  \frac{1}{(1-t)^4}\EE\left[\norm{\norm{Y_1}^2Y_1^{\otimes 2}}_{HS}^2\right] \leq\frac{1}{(1-t)^4}\EE\left[\norm{Y_1}^8\right].$$
	For an isotropic log concave measure, the expression $\EE\left[\norm{Y_1}^8\right]$ is bounded from above by
	$Cd^4$ for a universal constant $C > 0$ (see \cite{paouris2006concentration}). Thus,
	$$\mathrm{Tr}\left(\EE\left[\left(\Gamma_t^2 - \EE\left[\Gamma_t^2\right]\right)^2\right]\right) = \EE\left[\norm{\Gamma_t^2 - \EE\left[\Gamma_t^2\right]}_{HS}^2\right]\leq 2\EE\left[\norm{\Gamma_t^2}_{HS}^2\right] \leq C\frac{d^4}{(1-t)^4}.$$
\end{proof}
\begin{lemma} \label{sigma bounds}
	Suppose that $\mu$ is log concave and isotropic, then there exists a universal constant $1 > c > 0$ such that
	\begin{enumerate}
		\item For any, $t \in [0,\frac{c}{d^2}]$, $\sigma_t \geq \frac{1}{2}$.
		\item For any, $t \in [\frac{c}{d^2}, 1]$, $\sigma_t \geq \frac{c}{t d^2}$.
	\end{enumerate}
\end{lemma}
\begin{proof}
	By Lemma \ref{derivative of gamma}, we have
	$$\frac{d}{dt}\EE\left[\mathrm{Cov}(Y_1|\mathcal{F}_t)\right] =  -\EE\left[\Gamma_t^2\right] \stackrel{\eqref{gamma_var}}{=} - \frac{\EE\left[\mathrm{Cov}\left(Y_1|\mathcal{F}_t\right)^2\right]}{(1-t)^2}.$$
	Moreover, by convexity,
	$$\EE\left[\mathrm{Cov}\left(Y_1|\mathcal{F}_t\right)^2\right] \preceq
	 \EE\left[\EE\left[Y_1^{\otimes 2} |\mathcal{F}_t\right]^2\right] \preceq
	  \EE\left[\norm{Y_1}^4\right]\mathrm{I}_d.$$
	It is known (see \cite{paouris2006concentration}) then when $\mu$ is log concave and isotropic there exists a universal constant $C  > 0$ such that
	$$\EE\left[\norm{Y_1}^4\right] \leq Cd^2.$$
	Consequently, $\frac{d}{dt}\EE\left[\mathrm{Cov}(Y_1|\mathcal{F}_t)\right] \succeq -\frac{Cd^2}{(1-t)^2}\mathrm{I}_d$, and since $\mathrm{Cov}(Y_1|\mathcal{F}_0) = \mathrm{I}_d,$
	$$\EE\left[\mathrm{Cov}(Y_1|\mathcal{F}_t)\right] \succeq\left(1 - Cd^2\int\limits_0^t\frac{1}{(1-s)^2}ds\right)\mathrm{I}_d = \left(1 - \frac{Cd^2t}{1-t}\right)\mathrm{I}_d.$$
	By increasing the value of $C$, we may legitimately assume that $\frac{1}{Cd^2}\leq 1$, thus for any $t \in [0, \frac{1}{3Cd^2}]$ we get that $$\EE\left[\mathrm{Cov}(Y_1|\mathcal{F}_t)\right] \succeq \frac{1}{2}\mathrm{I}_d,$$ which implies $\sigma_t \geq \frac{1}{2}$ and completes the first part of the lemma.
	In order to prove the second part, we first write
	\begin{align} \label{gamma derivative}
	\frac{d}{dt}\EE\left[\Gamma_t\right] &= \frac{d}{dt}\frac{\EE\left[\mathrm{Cov}(Y_1|\mathcal{F}_t)\right]}{1-t} \stackrel{\mbox{ \tiny (Lemma \ref{derivative of gamma}) }}{=} \frac{\EE\left[\mathrm{Cov}(Y_1|\mathcal{F}_t)\right] - (1-t)\EE\left[\Gamma_t^2\right]}{(1-t)^2}
	= \frac{\EE\left[\Gamma_t\right] - \EE\left[\Gamma_t^2\right]}{1-t}.
	\end{align}
	Since, by Lemma \ref{gamma_properties}, $\Gamma_t \preceq \frac{1}{t}\mathrm{I}_d$, we have the bound
	$$\frac{\EE\left[\Gamma_t\right] - \EE\left[\Gamma_t^2\right]}{1-t}
	\succeq \frac{1 - \frac{1}{t}}{1-t}\EE\left[\Gamma_t\right] = - \frac{1}{t} \EE\left[\Gamma_t\right].$$
	Now, consider the differential equation $f'(t) = \frac{- f(t)}{t}$, $f\left(\frac{1}{3Cd^2}\right) = \frac{1}{2}$. Its unique solution is $f(t) = \frac{1}{6Cd^2t}$. Thus, Gromwall's inequality shows that $\sigma_t \geq \frac{1}{6Cd^2t}$, which concludes the proof.
\end{proof}
\begin{proof}[Proof of Theorem \ref{thm:ent-log-concave}]

Our objective is to bound from above the right hand side of Equation \eqref{entropy integral bonund}. As a consequence of Lemma \ref{sigma bounds}, we have that for any $t \in [0,1)$,
$$\int\limits_t^1\sigma_s^{-2}ds \leq Cd^4(1-t),$$
for some universal constant $C > 0$. It follows that the integral in \eqref{entropy integral bonund} admits the bound
$$\int\limits_0^{1}\frac{\EE\left[\mathrm{Tr}\left(\left(\Gamma_t^2 - \EE\left[\Gamma_t^2\right]\right)^2\right)\right]}{(1-t)^2\sigma_t^2}\left(\int\limits_t^1\sigma_s^{-2}ds\right)dt \leq Cd^4\int\limits_0^{1}\frac{\EE\left[\mathrm{Tr}\left(\left(\Gamma_t^2 - \EE\left[\Gamma_t^2\right]\right)^2\right)\right]}{(1-t)\sigma_t^2}dt.$$
Next, there exists a universal constant $C' > 0$ such that
$$Cd^4\int\limits_0^{cd^{-2}}\frac{\EE\left[\mathrm{Tr}\left(\left(\Gamma_t^2 - \EE\left[\Gamma_t^2\right]\right)^2\right)\right]}{(1-t)\sigma_t^2}dt \leq C'\int\limits_0^{cd^{-2}}  \frac{d^8}{(1-t)^5}dt \leq C'd^{8},$$
where we have used the second bound of Lemma \ref{trace bounds} and the first bound of Lemma \ref{sigma bounds}.
Also, by applying the second bound of Lemma \ref{sigma bounds} when $t \in [cd^{-2}, d^{-1}]$ we get
$$Cd^4\int\limits_{cd^{-2}}^{d^{-1}}\frac{\EE\left[\mathrm{Tr}\left(\left(\Gamma_t^2 - \EE\left[\Gamma_t^2\right]\right)^2\right)\right]}{(1-t)\sigma_t^2}dt \leq C'\int\limits_{cd^{-2}}^{d^{-1}} \frac{d^{12}t^2}{(1-t)^5}dt \leq C'd^{9}.$$
Finally, when $t > d^{-1}$, we have
\begin{align*}
Cd^4\int\limits_{d^{-1}}^{1}\frac{\EE\left[\mathrm{Tr}\left(\left(\Gamma_t^2 - \EE\left[\Gamma_t^2\right]\right)^2\right)\right]}{(1-t)\sigma_t^2}dt  &\leq C'd^8\int\limits_{d^{-1}}^1\frac{t^2\EE\left[\mathrm{Tr}\left(\left(\Gamma_t^2 - \EE\left[\Gamma_t^2\right]\right)^2\right)\right]}{1-t}dt \\
&\leq 2C'd^9\int\limits_{d^{-1}}^1 \left ( \frac{1}{t^2} + \EE\left[\norm{v_t}^2\right] \right )dt\\
 & \stackrel{\eqref{variational entropy}}{\leq} 4C'd^{10}(1 + \mathrm{Ent}(Y_1||G)),
\end{align*}
where the first inequality uses Lemma \ref{sigma bounds} and the second one uses Lemma \ref{trace bounds}. This establishes 
 \[ \Ent(S_n||G) \leq \frac{C d^{10}(1 +  \Ent(Y_1||G))}{n}. \]
\end{proof}
Finally, we derive an improved bound for the case of $1$-uniformly log concave measures, based on the following estimates.
\begin{lemma} \label{lem:uniformly}
	Suppose that $\mu$ is $1$-uniformly log concave, then for every $t \in [0,1)$
	\begin{enumerate}
		\item $\mathrm{Tr}\left(\EE\left[\left(\Gamma_t^2 - \EE\left[\Gamma_t^2\right]\right)^2\right]\right) \leq 2(1-t)\left(d - \mathrm{Tr}\left(\Sigma\right) + \EE\left[\norm{v_t}^2\right]\right).$
		\item $\sigma_t \geq \sigma_0$.
	\end{enumerate}
\end{lemma}
\begin{proof}
	By Lemma \ref{gamma_properties}, we have that $\Gamma_t \preceq \mathrm{I}_d$ almost surely. Using this together with the identity given by Lemma \ref{lem:gamma representation}, and proceeding in similar fashion to Lemma \ref{trace bounds} we obtain
	$$\mathrm{Tr}\left(\EE\left[\Gamma_t^2\right]\right) \geq \frac{1}{d} \mathrm{Tr}\left(\EE\left[\Gamma_t\right]\right)^2 \geq d - 2(1-t)\left(d - \mathrm{Tr}\left(\Sigma\right)+\EE\left[\norm{v_t}^2\right]\right),$$
	and
	\begin{align*}
	\mathrm{Tr}\left(\EE\left[\left(\Gamma_t^2 - \EE\left[\Gamma_t^2\right]\right)^2\right]\right) &\leq \mathrm{Tr}\left(\EE\left[\left(\Gamma_t^2 - \mathrm{I}_d\right)^2\right]\right) \leq \mathrm{Tr}\left(\EE\left[\mathrm{I}_d - \Gamma_t^2\right]\right)\\
	&\leq 2(1-t)\left(d - \mathrm{Tr}\left(\Sigma\right) + \EE\left[\norm{v_t}^2\right]\right).
	\end{align*}
	Also, recalling \eqref{gamma derivative} and since $\Gamma_t \preceq \mathrm{I}_d$ we get
	$$\frac{d}{dt}\EE\left[\Gamma_t\right] = \frac{\EE\left[\Gamma_t\right] - \EE\left[\Gamma_t^2\right]}{1-t} \geq 0,$$
	which shows that $\sigma_t$ is bounded from below by a non-decreasing function and so $\sigma_t \geq \sigma_0$ which is the minimal eigenvalue of $\Sigma$.
\end{proof}
\begin{proof}[Proof of Theorem \ref{thm:ent-strong-log-concave}]
	Plugging the bounds given in Lemma \ref{lem:uniformly} into Equation \eqref{entropy integral bonund} yields
	\begin{align*}
	\mathrm{Ent}(S_n||G) &\leq \frac{1}{n}\int\limits_0^1\frac{\EE\left[\mathrm{Tr}\left(\left(\Gamma_t^2 - \EE\left[\Gamma_t^2\right]\right)^2\right)\right]}{(1-t)^2\sigma_t^2}\left(\int\limits_t^1\sigma_s^{-2}ds\right)dt \\
	&\leq \frac{2\left(d + \int\limits_0^1\EE\left[\norm{v_t}^2\right]dt\right)}{\sigma_0^4n} \stackrel{\eqref{variational entropy}}{=} \frac{2\left(d + 2\mathrm{Ent}\left(X||\gamma\right)\right)}{\sigma_0^4n},
	\end{align*}
	which completes the proof.
\end{proof}

\bibliographystyle{plain}
\bibliography{bib}

\end{document}